\documentclass{amsart}
\usepackage{amssymb, amsmath, amsthm, graphics, comment, xspace, enumerate}
\usepackage{hyperref}
\usepackage{color}
\newtheorem{thm}{Theorem}[section]

\newtheorem{lem}[thm]{Lemma}

\newtheorem{prop}[thm]{Proposition}
\newtheorem{prob}[thm]{Problem}
\theoremstyle{definition}
\newtheorem{defn}[thm]{Definition}
\newtheorem{example}[thm]{Example}
\theoremstyle{remark}
\newtheorem{rem}[thm]{Remark}
\numberwithin{equation}{section}

\begin{document}
\title[Multiparameter $C$-semigroups and  multiparameter $C$-cosine...]{Multiparameter $C$-semigroups and  multiparameter $C$-cosine functions}

\author{Marko Kosti\' c}
\address{Faculty of Technical Sciences,
University of Novi Sad,
Trg D. Obradovi\' ca 6, 21125 Novi Sad, Serbia}
\email{marco.s@verat.net}

\author{Halis Can Koyuncuo\u{g}lu}
\address{Department of Engineering,
Sciences\\Izmir Katip Celebi University, 35620, Cigli, Izmir, Turkey}
\email{haliscan.koyuncuoglu@ikcu.edu.tr}

\author{Youssef N. Raffoul}
\address{Department of Mathematics,
University of Dayton,
Dayton, OH 45469-2316 USA}
\email{yraffoul1@udayton.edu}

{\renewcommand{\thefootnote}{} \footnote{
\\ \text{  }  \ \    {\it Key words and phrases.} Multiparameter $C$-semigroups, multiparameter $C$-cosine functions, abstract multiparameter Cauchy problems, automatic extensions, locally convex spaces.}}

\begin{abstract}
In this paper, we present several new results concerning multiparameter $C$-semigroups. We introduce and systematically analyze the class of multiparameter $C$-cosine functions, providing several new structural results and applications to abstract multiparameter Cauchy problems of first/second order in locally convex spaces. We also consider automatic extensions of multiparameter $C$-semigroups and multiparameter $C$-cosine functions. 
\end{abstract}
\maketitle

\section{Introduction and preliminaries}

The class of multiparameter semigroups of operators was introduced by E. Hille in 1944; see \cite{bberens} and \cite{hill} for basic information in this direction. 
If $(X,\| \cdot \|)$ is a Banach space, then by a multiparameter semigroup we mean any
operator-valued function $T : [0,+\infty)^{n} \rightarrow L(X)$ such that $T(0)={\rm I}$ and $T(t+s)=T(t)T(s)$ for all $t,\ s\in [0,+\infty)^{n};$ here, $L(X)$ denotes the Banach space of all bounded linear operators on $X$ and ${\rm I}$ denotes the identity operator on $X$. A semigroup  $(T(t))_{t\in [0,+\infty)^{n}}$ is called strongly continuous if the mapping $t\mapsto T(t)x,$ $t\in [0,+\infty)^{n}$ is strongly continuous at $t=0.$ 

The class of $C$-semigroups and its applications to ill-posed evolution equations, where the operator $C\in L(X)$ is injective, were analyzed by R. deLaubenfels in the important research monograph \cite{l1} (1994).
Further on, the class of two-parameter $C$-semigroups defined on the set $[0,+\infty)^{2}$ and the related Cauchy problems were analyzed by M. Janfada in 2009 (\cite{janfada}); cf. also the recent research article \cite{blali} by A. Blali, A. El Amrani and J. Ettayb for the notion of a
two-parameter $C$-groups of bounded linear operators
on non-Archimedean Banach spaces, and the research article \cite{Khanehgir}, where
M. Khanehgir, F. Roohei and E. Goharei Moghaddam have analyzed two-parameter $N$-times integrated semigroups. 

The class of local $C$-semigroups and the class of local integrated semigroups were investigated in the papers \cite{tanaka} by N. Tanaka, N. Okazawa (1990) and \cite{a42} by W. Arendt, O. El--Mennaoui, V. Keyantuo (1994).
In this paper, we introduce and analyze the class of multiparameter $C$-semigroups defined on arbitrary non-empty set $\Omega \subseteq {\mathbb R}^{n}$ containing zero. In our approach, a strongly continuous family $(T(t))_{t\in \Omega}\subseteq L(X)$ 
is said to be a (local, if $\Omega$ is bounded) $C$-semigroup if $T(t+s)C=T(t)T(s)$ for all $t$, $s\in \Omega$ with $t+s \in \Omega$,
and $T(0)=C;$ here, the operator $C\in L(X)$ need not be injective and $X$ is a general Hausdorff sequentially complete locally convex space over the field of complex numbers (SCLCS). The notion of a $C$-semigroup $(T(t))_{t\in \Omega}$ extends the notion introduced by M. Janfada in \cite[Definition 2.1]{janfada}, where the author has assumed that the operator $C$ is injective and $\Omega=[0,+\infty)^{2}.$  For further information concerning multiparameter strongly continuous semigroups, multiparameter integrated $C$-semigroups and their applications, we refer the reader to the list of references quoted on pp. 593--595 in \cite{nova-selected}.  

The class of cosine functions was introduced by the Czech mathematician M. Sova (\cite{sova}; 1966) with a view to study the abstract differential equations of second order; it is also worth noticing that some initial results in this direction were obtained by S. Kurepa in \cite{kurepa}-\cite{skurepa} (1960-1962). Recall, a strongly continuous family $(C(t))_{t\in {\mathbb R}}$ of bounded linear operators on $X$ 
is said to be a cosine function if  $C(t+s)+C(t-s)=2C(t)C(s)$ for all $t$, $s\in {\mathbb R}$ and $C(0)={\rm I}.$ The infinitesimal generator $A$ of $(C(t))_{t\in {\mathbb R}}$ is defined by 
$$
D(A):=\Bigl\{x \in X : \lim_{t\rightarrow 0}2t^{-2}\bigl( C(t)x-x\bigr) \mbox{ exists in }X\Bigr\}
$$
and
$$
Ax:=\lim_{t\rightarrow 0}2t^{-2}\bigl( C(t)x-x\bigr) ,\quad x\in D(A).
$$
It is a closed linear operator which is densely defined in $X.$
By \cite[Theorem 1]{skurepa}, the Lebesgue measurability of function $t\mapsto C(t) \in L(X)$ is equivalent with its continuity; if this is the case, the infinitesimal generator $A$ of 
$(C(t))_{t\in {\mathbb R}}$ is a bounded linear operator on $X$, and we have
\begin{align}\label{cc}
C(t)=\sum_{k=0}^{+\infty}\frac{t^{2k}A^{k}}{(2k)!},\quad t\in  {\mathbb R},
\end{align}
with the convergence being absolute and uniform on compact subsets of ${\mathbb R}$.

The class of local $C$-cosine functions was introduced by 
F. L. Huang and
T. W. Huang (\cite{huang}, 1995). For more details about (local) $C$-semigroups and (local) $C$-cosine functions, we refer the reader to the research monographs \cite{l1}, \cite{fadoa}, \cite{knjiga}-\cite{FKP}, \cite{mli}, the research articles \cite{ksho11}-\cite{ksho12}, \cite{lishaw2007}, \cite{pshaw}, \cite{shaw2005} and the references quoted therein; for the theory of local $C$-cosine functions and their applications, we would like to specifically mention the short monograph \cite{tajw} by S.-Y. Shaw (2002).

The class of multidimensional $C$-cosine functions has not been introduced in the existing literature by now. In connection with this issue, we would like to emphasize the following facts: 
If $B$ is a complex Banach algebra with unit element $e,$ then the properties of mapping $F: X\rightarrow X$ which satisfies the cosine functional equation
\begin{align}\label{cosine}
F(x+y)+F(x-y)=2F(x)F(y),\quad x,\ y\in X;\ \ F(0)=e
\end{align}
were investigated by S. Kurepa in \cite[Section 3]{skurepa}. In particular, if for each $x\in X$ the mapping $t\mapsto F(tx)\in B,$ $t\in {\mathbb R}$ is Lebesgue measurable and \eqref{cosine} holds, then we know, by \cite[Theorem 3]{skurepa}, that there exists a mapping $A: X \rightarrow B$ such that the following conditions hold:
\begin{itemize}
\item[(i)] $F(x)=\sum_{k=0}^{+\infty}\frac{[A(x)]^{k}}{(2k)!},\ x\in  X;$ $A(tx)=t^{2}A(x),\ t\in {\mathbb R},\ x\in X;$
\item[(ii)] $A(x)A(y)=A(y)A(x)\mbox{ and }A(x+y)+A(x-y)=2A(x)+2A(y),\ x,\ y\in X;$
\item[(iii)] $A^{2}(x+y)+A^{2}(x-y)=2A^{2}(x)+2A(y)+12A(x)A(y),\ x,\ y\in X;$
\item[(iv)] $A(\cdot)$ is continuous if and only if $\lim_{t\rightarrow 0}F(tx)=e,$ uniformly for $x$ in some sphere of $X;$
\item[(v)] if $x_{0}\in X$ satisfies $\|e-A(x_{0}) \|<1,$ then there exists an additive function $L: X \rightarrow B$ such that $A(x)=[L(x)]^{2}$ for all $x\in X;$ moreover, the continuity of $A(\cdot)$ implies the continuity of $L(\cdot).$
\end{itemize}
Concerning the cosine functional equation, the interested reader may also consult \cite[Theorem 3, Theorem 5]{kurepa}, as well as the research articles \cite{fris} by  
P. de Place Friis, H. Stetker, \cite{kana} by P. Kannappan, as well as the research monographs \cite{ace} by 
J. Acz\'el, J. Dhombres, \cite{cer} by S. Czerwik and \cite{stet} by H. Stetker.

In this paper, we continue the analysis of cosine functional equation by exploring the class of multiparameter $C$-cosine functions defined on arbitrary non-empty set $\Omega \subseteq {\mathbb R}^{n}$ containing zero. In our approach, a strongly continuous family $(C(t))_{t\in \Omega}\subseteq L(X)$ 
is said to be a (local, if $\Omega$ is bounded) $C$-cosine function if $C(t+s)C+C(t-s)C=2C(t)C(s)$ for all $t$, $s\in \Omega$ with $t\pm s \in \Omega$,
and $C(0)=C;$ here, the operator $C \in L(X)$ need not be injective and $X$ is an SCLCS. 

The organization and main ideas of paper can be described as follows. 
Suppose that $(a_{1},...,a_{n})$ is a basis in ${\mathbb R}^{n}$, $\tau_{i}\in (0,+\infty]$ for $1\leq i \leq n$, $D_{i}=[0,\tau_{i})$ or $D_{i}=(-\tau_{i},\tau_{i})$ for $1\leq i\leq n,$ and
\begin{align}\label{me1}
\Omega :=\Bigl\{ t_{1}a_{1}+...+t_{n}a_{n} : t_{i}\in D_{i}\ \ (1\leq i\leq n) \Bigr\}.
\end{align}
If $(T(t))_{t\in \Omega}$ is a $C$-semigroup and $1\leq i\leq n$, then we define the directional infinitesimal generator ${\mathcal A}_{i}$ of $(T(t))_{t\in \Omega}$ by its graph
\begin{align}\label{maww1}
{\mathcal A}_{i} :=\Biggl\{ ( x,y) \in X \times X : \lim_{t_{i}\rightarrow 0}\frac{T\bigl(0,...,t_{i}a_{i},...,0\bigr)x-Cx}{t_{i}}=Cy\Biggr\}.
\end{align}
In this paper, we consider the well-posedness of the following
$n$-parameter abstract Cauchy inclusion
\[ (ACP)_{\Omega}: \left\{
\begin{array}{l}
u\in C(\Omega : X)\cap C^1 \bigl(\Omega^{\circ} :X\bigr),\\
u_{t_{i};a_{i}}(t)\in {\mathcal A}_{i}u(t),\;t \in \Omega^{\circ},\ 1\leq i\leq n,\\
u(0)=Cx,
\end{array}
\right. \]
where
$$
u_{t_{i};a_{i}}(t):= \lim_{h_{i}\rightarrow 0}\frac{u\bigl(t+h_{i}a_{i} \bigr)-u(t)}{h_{i}} ,\quad t\in \Omega,
$$
provided that the above limit exists in $X;$ cf. Theorem \ref{tucko1} for more details (in the formulation of this result, we also clarify some other structural properties of multiparameter $C$-semigroups).

Section \ref{ccos} is devoted to the study of multiparameter $C$-cosine functions. The basic notion 
is introduced in Definition \ref{1.2.1.1.1}; after that, we prove that the assumptions $\Omega=-\Omega,$  $\Omega$ is convex and $\Omega$ has a nonempty interior together with the assumption that the operator $C$ is injective imply that a $C$-cosine function $(C(t))_{t\in \Omega}$ is commutative, i.e., $C(t)C(s)=C(s)C(t)$ for all $t,\ s\in \Omega ;$ cf. Proposition \ref{rambo}. Some similar results are clarified in Proposition \ref{df}. Further on, in Proposition \ref{cos}, we prove that any cosine function $(C(t))_{t\in [0,+\infty)^{n}}$ on a complex Banach space $X$ is exponentially bounded; cf. also Remark \ref{renik} and Proposition \ref{asd}, 
where we observe that it is very difficult to precisely compute the multidimensional Laplace transform of an exponentially equicontinuous $C$-cosine function $(C(t))_{t\in [0,+\infty)^{n}}$ if $n\geq 2$. 

Abstract multiparameter second-order Cauchy problems are considered in Subsection \ref{gsd}. The first structural result of this subsection is Theorem \ref{emoj}, where we prove that the expression 
\begin{align*} 
C(t):=\frac{1}{2}\Bigl[ T(t)+T(-t)\Bigr],\quad t\in \Omega,
\end{align*}
determines a $C$-cosine function in the case that $\Omega=-\Omega$ and $(T(t))_{t\in \Omega}$ is a $C$-semigroup. If $\tau_{i}\in (0,+\infty]$ for $1\leq i \leq n$, $\Omega =(-\tau_{1},\tau_{1}) \times ... \times (-\tau_{n},\tau_{n})$, ${\mathcal A}_{i}$ is a subgenerator of a $C$-semigroup $(T_{i}(t_{i}) \equiv T(t_{i}e_{i}))_{0\leq \tau <\tau_{i}}$ 
and
${\mathcal A}_{i}^{2}$ is closed ($1\leq i\leq n$), then Theorem \ref{emoj} also asserts that, for every $x\in \bigcap_{1\leq i \leq n}D({\mathcal A}_{i}^{2}),$ the function $u(t):=C(t)C^{n-1}x,$ $t\in \Omega$ is a solution of the following $n$-parameter Cauchy inclusion of second order
\[ (ACP_{2}): \left\{
\begin{array}{l}
u\in C^{2}(\Omega : X);\\ \frac{\partial ^{2}u}{\partial t_{i}^{2}}(t)\in {\mathcal A}_{i}^{2}u(t),\;t \in \Omega,\ 1\leq i\leq n;\\
u(0)=C^{n}x;\\  \Bigl( \frac{\partial u}{\partial t_{i}}\Bigr)_{t=(t_{1},...,t_{i-1},0,t_{i+1},...,t_{n})} \in \frac{1}{2}{\mathcal A}_{i}\Biggl[       
\prod_{\substack{1\le j\le n\\ i\ne j}} T_{j}\bigl( t_{j}\bigr)x-\prod_{1\leq j\leq n\atop i\neq j} T_{j}\bigl( -t_{j}\bigr)x\Biggr],\\ t\in \Omega,\ 1\leq i\leq n,
\end{array}
\right. \]
as well as that the uniqueness of solutions to $(ACP_{2})$ holds if the operator $C$ is injective.

The second structural result of Subsection \ref{gsd} is Theorem \ref{novio}. If $\tau_{i}\in (0,+\infty]$ for $1\leq i \leq n$, $\Omega =(-\tau_{1},\tau_{1}) \times ... \times (-\tau_{n},\tau_{n})$ and $(C(t))_{t\in \Omega}$ 
is a $C$-cosine function, then we first prove here that, if $1\leq i\leq n$, the operator ${\mathcal A}_{i}$ is the integral generator of a $C$-cosine function
$(C_{i}(t_{i})\equiv C(t_{i}e_{i}))_{-\tau_{i}< \tau <\tau_{i}},$ $C(t)C(s)=C(s)C(t)$ for all $t,\ s\in \Omega$ and $x\in D({\mathcal A}_{i}),$ then the function 
\begin{align}\label{dfg}
u_{i}(t):=C\bigl(t_{1},...,t_{i-1},t_{i},t_{i+1},...,t_{n} \bigr)Cx+C\bigl(-t_{1},...,-t_{i-1},t_{i},-t_{i+1},...,-t_{n} \bigr)Cx,\ t\in \Omega
\end{align}
 is a solution of the following $n$-parameter Cauchy inclusion of second order
\[ (ACP_{2;i}): \left\{
\begin{array}{l}
u_{i}\in C^{2}(\Omega : X);\\ \frac{\partial ^{2}u_{i}}{\partial t_{i}^{2}}(t)\in {\mathcal A}_{i}u_{i}(t),\;t \in \Omega;\\ u_{i}
\bigl(t_{1},...,t_{i-1},0,t_{i+1},...,t_{n} \bigr)\\=C\bigl(t_{1},...,t_{i-1},0,t_{i+1},...,t_{n} \bigr)Cx+C\bigl(-t_{1},...,-t_{i-1},0,-t_{i+1},...,-t_{n} \bigr)Cx, \\ \mbox{ if }t= \bigl(t_{1},...,t_{i-1},0,t_{i+1},...,t_{n} \bigr) \in \Omega ;\\
\Bigl( \frac{\partial u_{i}}{\partial t_{i}}\Bigr)_{t=(t_{1},...,t_{i-1},0,t_{i+1},...,t_{n})}=0,\quad \mbox{ if }t=\bigl(t_{1},...,t_{i-1},0,t_{i+1},...,t_{n} \bigr) \in \Omega .
\end{array}
\right. \]

If $n=2$ and $x\in D({\mathcal A}_{1}) \cap D({\mathcal A}_{2}),$ then we prove that $u_{1}=u_{2}=:u$ is a solution of the following two-parameter Cauchy inclusion of second order
\[ (ACP_{2})': \left\{
\begin{array}{l}
u\in C^{2}(\Omega : X);\\ \frac{\partial ^{2}u}{\partial t_{i}^{2}}(t)\in {\mathcal A}_{i}u(t),\;t \in \Omega,\ 1\leq i\leq 2;\\ u\bigl(t_{1},0\bigr)=2C\bigl(t_{1},0\bigr)Cx,\mbox{ if }|t_{1}|<\tau_{1}, \ u\bigl(0,t_{2}\bigr)=2C\bigl(0,t_{2}\bigr)Cx,\mbox{ if }|t_{2}|<\tau_{2}; \\
\Bigl( \frac{\partial u}{\partial t_{1}}\Bigr)_{t=(0,t_{2})}=0,\mbox{ if }|t_{2}|<\tau_{2},\ \ \mbox{ and }\Bigl( \frac{\partial u}{\partial t_{2}}\Bigr)_{t=(t_{1},0)}=0,\mbox{ if }|t_{1}|<\tau_{1}.
\end{array}
\right. \]
Furthermore, if the operator $C$ is injective, then ${\mathcal A}_{1}$ and ${\mathcal A}_{2}$ are closed linear operators and $u(t_{1},t_{2}):=C(t_{1},t_{2})Cx+C(t_{1},-t_{2})Cx,$ $t=(t_{1},t_{2})\in \Omega$ is a unique solution of problem $(ACP_{2})'.$ 

Finally, if $n \geq 2$ and $x\in D({\mathcal A}_{1}) \cap D({\mathcal A}_{2}) \cap ... \cap D({\mathcal A}_{n})$, then we clarify that the function
\begin{align}
u\bigl(t_{1},..., t_{n}\bigr):=\sum_{(\sigma_{1},...,\sigma_{n}) \in \{ -1,1\}^{n}}C\bigl( \sigma_{1}t_{1},..., \sigma_{n}t_{n}\bigr)Cx,\quad t=\bigl(t_{1},..., t_{n}\bigr) \in \Omega,
\end{align}
is a solution of the following $n$-parameter Cauchy inclusion of second order
\[ (ACP_{2})'': \left\{
\begin{array}{l}
u\in C^{2}(\Omega : X);\\ \frac{\partial ^{2}u}{\partial t_{i}^{2}}(t)\in {\mathcal A}_{i}u(t),\;t \in \Omega , \ 1\le i\leq n;\\ u(0)=2^{n}C^{2}x; \\
\Bigl( \frac{\partial u}{\partial t_{i}}\Bigr)_{t=(t_{1},...,t_{i-1},0,t_{i+1},...,t_{n})}=0,\ \mbox{if }t=\bigl(t_{1},...,t_{i-1},0,t_{i+1},...,t_{n} \bigr) \in \Omega ,
\end{array}
\right. \]
as well as that the injectiveness of operator 
$C$ implies that ${\mathcal A}_{i}$ is a closed linear operator for $1\leq i \leq n$ and the function $u(\cdot)$, given above, is a unique solution of problem $(ACP_{2})''.$  

In the one-dimensional setting, the extensions of local (convoluted) $C$-semigroups and local (convoluted) $C$-cosine functions have been investigated by S. W. Wang and M. C. Gao \cite[Theorem 2.6, Theorem 3.6]{wang}, provided that the operator $C$ is injective, and M. Kosti\' c \cite[Theorem 3.2.60, Remark 3.2.61]{FKP}, in general case.  Section \ref{subc} investigates the automatic extensions of multiparameter $C$-semigroups and multiparameter $C$-cosine functions; the main structural result of this section is Theorem \ref{nap}, which asserts the following: If the operator $C$ is injective, $\tau_{i}\in (0,+\infty]$ for $1\leq i \leq n$, $\Omega =[0,\tau_{1}) \times ... \times [0,\tau_{n}),$ resp. $\Omega =(-\tau_{1},\tau_{1}) \times ... \times (-\tau_{n},\tau_{n}),$ and $(T(t))_{t\in \Omega}$ 
is a locally equicontinuous $C$-semigroup, resp. $(C(t))_{t\in \Omega}$ 
is a locally equicontinuous $C$-cosine function given by the equality \eqref{3456} below, then, for every integer $k\in {\mathbb N},$ there exists a locally equicontinuous $C^{k}$-semigroup $(T_{k}(t))_{t\in k\Omega}$ such that $T_{k}(t)=T(t)C^{k-1}$ for all $t\in \Omega,$ resp. there exists a locally equicontinuous $C^{k}$-cosine function $(C_{k}(t))_{t\in k\Omega}$ such that $C_{k}(t)=C(t)C^{k-1}$ for all $t\in \Omega.$ We also propose an open question concerning the automatic extension of a general locally equicontinuous $C$-cosine function $(C(t))_{t\in \Omega}$; cf. Problem \ref{proba}.

In Section \ref{separ}, which can be viewed of some independent interest, we consider the existence and uniqueness of asymptotically almost periodic solutions to abstract multiparameter Cauchy problems (the results established in this section will appear in the forthcoming research monograph \cite{apsclcs}). In order to better explain our strivings, let us consider the operator family $(R(t))_{t\in [0,+\infty)^{n}}$ given by 
\begin{align}\label{prod} R(t):=R_{1}\bigl(t_{1}\bigr)\cdot...\cdot R_{n}\bigl(t_{n}\bigr),\quad t=\bigl(t_{1},...,t_{n}\bigr)\in [0,+\infty)^{n},
\end{align}
where $(X,\| \cdot \| )$ is a complex Banach space and the following condition holds:
\begin{itemize}
\item[(D1)] $\alpha_{i}>0$ and ${\mathcal A}_{i}$ is a subgenerator of a (local) $(g_{\alpha_{i}},C_{i})$-regularized resolvent family $(R_{i}(t_{i}))_{0\leq t_{i}<\tau_{i}}$ ($1\leq i\leq n$), $R_{i}(t_{i})R_{j}(t_{j})= R_{j}(t_{j})R_{i}(t_{i})$ and $C_{i}R_{j}(t_{j})=R_{j}(t_{j})C_{i}$ for all $1\leq i, \ j\leq n $ such that $i\neq j,$ $0\leq t_{i}<\tau_{i}$ and $0\leq t_{j}<\tau_{j}.$ Here, $g_{\zeta}(t):=t^{\zeta-1}/\Gamma(\zeta),$ $t>0$ ($\zeta>0 $), where $\Gamma(\cdot)$ denotes the Gamma function.
\end{itemize}
Every multiparameter strongly continuous semigroup $(T(t))_{t\in [0,+\infty)^{n}}$ can be represented in this way; cf. also Theorem \ref{tucko1} and the equality \eqref{345} for the corresponding results obtained for multiparameter $C$-semigroups.

In \cite[Theorem 8.1.23]{apsclcs}, we have recently proved that the assumption $x\in \bigcap_{1\leq i\leq n}D({\mathcal A}_{i} )$ implies that there exists a unique strong solution $u(\cdot) $ of problem
\[(P):\left\{
\begin{array}{l}
u\in {\mathcal C}(I  :X);\\
{\mathbf D}_{t_{i}}^{\alpha_{i}}u(t) \in {\mathcal A}_{i}u(t),\;t \in I,\ 1\leq i\leq n;\\
u(0)=C_{1}\cdot ....\cdot C_{n}x;\\ 
\Biggl[ \frac{\partial^{j}}{\partial t_{i}^{j}}u\bigl(t_{1},...,t_{i-1},t_{i},t_{i+1},...,t_{n}\bigr) \Biggr]_{t_{i}=0}=0
\mbox{ for all }  1\leq i\leq n,\ 0<j<m_{i}\mbox{ and }\\
\bigl(t_{1},...,t_{i-1},t_{i+1},...,t_{n}\bigr)\in [0,\tau_{1})\times ... \times [0,\tau_{i-1})\times [0,\tau_{i+1}) \times ... \times [0,\tau_{n}),
\end{array}
\right.
\]
which is given by
$
u(t):=R (t)x,\  t\in I;
$
here, $I=[0,\tau_{1}) \times ... \times [0,\tau_{n}) $, $m_{i}=\lceil \alpha_{i} \rceil$ for $1\leq i\leq n,$ and ${\mathcal C}(I: X)$ denotes the vector space consisting of all continuous functions $u : I\rightarrow  X$ such that the partial Caputo fractional derivative ${\mathbf D}_{t_{i}}^{\alpha_{i}}u(\cdot)$ is well-defined for $1\leq i\leq n.$ 
In the proofs of structural results concerning the abstract multiparameter Cauchy problems of second order, the uniqueness of solutions will be shown with the help of argumentation contained in the proof of the last mentioned result with $\alpha_{i}=2$ ($i\in {\mathbb N}_{n}$).

We first observe that the almost periodicity of each strongly continuous operator family $(R_{i}(t_{i}))_{t_{i}\geq 0}$ implies that $(R(t))_{t\in [0,+\infty)^{n}}$ is almost periodic, i.e., for each $x\in X,$ the mapping $t\mapsto R(t)x,$ $t\in [0,+\infty)^{n}$ is almost periodic, which means that,
for every $\epsilon>0$, there exists $L>0$ such that for each $t_{0}\in [0,+\infty)^{n}$ the ball $B(t_{0},L)\equiv \{t\in {\mathbb R}^{n} : |t-t_{0}|\leq L \},$ where $|\cdot - \cdot|$ denotes the Euclidean distance in ${\mathbb R}^{n}$, contains a point $\tau \in  [0,+\infty)^{n}$ such that, for every $t=(t_{1},...,t_{n})\in [0,+\infty)^{n}$, we have $\|R(t+\tau)x-R(t)x\| \leq \epsilon ;$ cf. \cite{nova-selected} for more details. After that, we consider the situation where
each operator family $(R_{i}(t_{i}))_{t_{i}\geq 0}$ is asymptotically almost periodic, i.e., for every $x\in X$ and $i\in {\mathbb N}_{n} ,$ the mapping $t_{i}\mapsto R_{i}(t_{i})x,$ $t_{i}\geq 0$ is asymptotically almost periodic in the sense that there exists an almost periodic function $h_{i} : [0,+\infty)\rightarrow X$ and a continuous function $q_{i} : [0,+\infty)\rightarrow X$ vanishing at plus infinity such that $R_{i}(t_{i})x=h_{i}(t)+q_{i}(t_{i}) $ for all $t_{i}\geq 0.$ Then we clarify that $(R(t))_{t\in [0,+\infty)^{n}}$ is asymptotically almost periodic in the sense that, for every $x\in X$ and $\epsilon>0$, there exist $L>0$ and $M>0$ such that for each $t_{0}\in [0,+\infty)^{n}$ the ball $B(t_{0},L)$ contains a point $\tau \in  [0,+\infty)^{n}$ such that, for every $t=(t_{1},...,t_{n})\in [0,+\infty)^{n}$ with the property that $\min(t_{1},...,t_{n})\geq M$, we have $\|R(t+\tau)x-R(t)x\| \leq \epsilon.$  Some illustrative applications of established results are presented, as well.\vspace{0.1cm}

\noindent {\bf Notation and preliminaries.} Unless stated otherwise, we will always assume henceforth that $X$ is an SCLCS and $n\in {\mathbb N};$ by $(e_{1},...,e_{n})$ we denote the standard basis of ${\mathbb R}^{n}.$ 
The abbreviation $\circledast$ denotes the fundamental system of seminorms\index{system of seminorms} which defines the topology of $X.$ For further information concerning the integration of functions with values in SCLCSs, the multivalued linear operators in sequentially complete locally convex spaces (MLOs) and solution operator families subgenerated by them, we refer the reader to \cite{FKP}; we will use the same notation as in this monograph. By ${\rm I}$ we denote the identity operator on $X.$ Given the numbers $s\in\mathbb R$ and $m\in {\mathbb N},$ we set $\lceil s\rceil:=\inf\{l\in\mathbb Z:s\leq l\}$ and ${\mathbb N}_{m}:=\{1,...,m\} .$ 
 
If $\emptyset \neq \Omega \subseteq {\mathbb R}^{n},$ then
a strongly continuous family $(R(t))_{t\in \Omega}\subseteq L(X)$ is said to be locally equicontinuous if, for every compact set $K\subseteq \Omega,$
the family $\{ R(t) : t\in \Omega \}$ is equicontinuous, i.e., for each seminorm $p\in \circledast$ there exist a real number $M>0$ and a seminorm $q\in \circledast$ such that $p(R(t)x)\leq cq(x)$ for all $x\in X$ and $t\in K.$ Let us emphasize that, if the pivot space $X$ is barreled, then $(R(t))_{t\in \Omega}\subseteq L(X)$ is always locally equicontinuous; this can be shown in the same way as in the proof of \cite[Proposition 1.1]{komura}. 

\section{Multiparameter $C$-semigroups}\label{sec1}

We start this section by introducing the following notion: 

\begin{defn}\label{1.2.1.1}
Let $0\in \Omega$ and $\Omega  \subseteq {\mathbb R}^{n}.$ A strongly continuous family $(T(t))_{t\in \Omega}\subseteq L(X)$ 
is said to be a (local, if $\Omega$ is bounded) $C$-semigroup if the following conditions hold:
\begin{itemize}
\item[(i)] $T(t+s)C=T(t)T(s)$ for all $t$, $s\in \Omega$ with $t+s \in \Omega$,
\item[(ii)] $T(0)=C$.
\end{itemize}

If $\Omega=[0,+\infty)^{n},$ then it is said that $(T(t))_{t\in [0,+\infty)^{n}}$ is exponentially equicontinuous if there exist real numbers $\omega_{1},...,\omega_{n}$ such that for each seminorm $p\in \circledast$ there exist a real number $M>0$ and a seminorm $q\in \circledast$ such that $p(T(t)x)\leq M\exp(\omega_{1}t_{1}+...+\omega_{n}t_{n})q(x)$ for all $x\in X$ and $t=(t_{1},...,t_{n})\in [0,+\infty)^{n} $; if this is the case and $(X,\| \cdot \|)$ is a complex Banach space, then we also say that  $(T(t))_{t\in [0,+\infty)^{n}}$ is exponentially bounded.

Finally, if $C={\rm I}, $ then we also say that $(T(t))_{t\in \Omega}$ is a strongly continuous semigroup.
\end{defn}

If $(T(t))_{t\in \Omega} $ 
is a $C$-semigroup, then plugging $s=0$ in (i.1) immediately implies that $T(t)C=T(t)C$ for all $t\in \Omega$; it is clear that we also have $T(s)T(t)=T(t)T(s)$ for all $t$, $s\in \Omega$ with $t+s \in \Omega$. If the operator $C$ is not injective, then 
the equality $T(s)T(t)=T(t)T(s)$ need not be satisfied for all $t$, $s\in \Omega$ (see, e.g., \cite[p. 1098]{lishaw2007} for the one-dimensional setting); if the region $\Omega$ has some specific geometrical properties, then we will later prove that the injectiveness of operator $C$ implies $T(s)T(t)=T(t)T(s)$ for all $t$, $s\in \Omega $ (see Theorem \ref{tucko1}(i)-(a)).

\begin{rem}\label{konc}
Let $n=1$ and $\Omega=[0,\tau)$ for some $0<\tau \leq +\infty .$
\begin{itemize}
\item[(i)] A multivalued linear operator ${\mathcal A}$ which satisfies 
\begin{itemize}
\item[(a)] $T(t){\mathcal A}\subseteq {\mathcal A}T(t)$, $t\in[0,\tau)$,
\item[(b)] $T(t)x-Cx=\int_0^tT(s)y\,ds$, $t\in[0,\tau)$, whenever $(x,y)\in {\mathcal A}$, 
\item[(c)] $(\int_0^tT(s)x\,ds,T(t)x-Cx)\in {\mathcal A}$ for all $x\in X$ and $t\in [0,\tau),$
\end{itemize}
is called a subgenerator of $(T(t))_{t\in[0,\tau)}$. In this case, the integral generator of $(T(t))_{t\in[0,\tau)}$ is defined by
\begin{align*}
&\hat{{\mathcal A}}:=\biggl\{(x,y)\in X\times X: T(t)x-Cx=\int_0^tT(s)y\,ds,\;t\in[0,\tau)\biggr\},
\end{align*}
and it coincides with the infinitesimal generator of $(T(t))_{t\in[0,\tau)}$ which is defined by
\begin{align*}
{\mathcal A}_{inf} :=\Biggl\{ ( x,y) \in X \times X : \lim_{t\rightarrow 0}\frac{T(t)x-Cx}{t}=Cy\Biggr\}.
\end{align*} 
\item[(ii)] If $C$ is injective, then $(T(t))_{t\in[0,\tau)}$ is non-degenerate in the sense that the assumption $T(t)x=0$ for all $t\in [0,\tau)$ implies $x=0$ and $\hat{{\mathcal A}}$ is single-valued; in this case, we also know that $\hat{{\mathcal A}}$ is the maximal subgenerator of $(T(t))_{t\in[0,\tau)}$ with respect to the set inclusion, and $C^{-1}\hat{{\mathcal A}}C=\hat{{\mathcal A}}.$ If the operator $C$ is not injective, then the integral generator $\hat{{\mathcal A}}$  of $(T(t))_{t\in[0,\tau)}$ need not be its subgenerator (see \cite[Example 3.2.42]{FKP}); furthermore, the local equicontinuity of a $C$-semigroup $(T(t))_{t\in [0,\tau)} $ implies that its integral generator $\hat{{\mathcal A}}$ is a closed MLO.
\end{itemize} 
\end{rem}

In what follows, we consider the $C$-semigroups defined on the region $\Omega$ of form \eqref{me1},
where $(a_{1},...,a_{n})$ is a basis in ${\mathbb R}^{n}$, $\tau_{i}\in (0,+\infty]$ for $1\leq i \leq n$, and $D_{i}=[0,\tau_{i})$ or $D_{i}=(-\tau_{i},\tau_{i})$ for $1\leq i\leq n.$ If $(T(t))_{t\in \Omega}$ is a $C$-semigroup and $i\in {\mathbb N}_{n}$, then we define the directional infinitesimal generator ${\mathcal A}_{i}$ of $(T(t))_{t\in \Omega}$ by \eqref{maww1}.
It is clear that for each $i\in {\mathbb N}_{n}$ the multivalued linear operator ${\mathcal A}_{i}$ is the infinitesimal generator of a $C$-semigroup $(T_{i}(t_{i})\equiv T(t_{i}a_{i}))_{t_{i}\in [0,\tau_{i})} $ as well as that ${\mathcal A}_{i}$ is single-valued if the operator $C$ is injective. The tuple ${\mathcal A}:=({\mathcal A}_{1},...,{\mathcal A}_{n})$ is said to be the directional infinitesimal generator of $(T(t))_{t\in \Omega}$.

Now we will state and prove the following result:

\begin{thm}\label{tucko1}
\begin{itemize}
\item[(i)] Suppose that $(a_{1},...,a_{n})$ is a basis in ${\mathbb R}^{n}$, $\tau_{i}\in (0,+\infty]$ for $1\leq i \leq n$, $D_{i}=[0,\tau_{i})$ or $D_{i}=(-\tau_{i},\tau_{i})$ for $1\leq i\leq n,$ 
and $\Omega$ is given by \eqref{me1}. If $(T(t))_{t\in \Omega} $ 
is a $C$-semigroup, then we have $T_{i}(t_{i})T_{j}(t_{j})=T_{j}(t_{j})T_{i}(t_{i})$ for all $1\leq i,\ j\leq n$ with $i\neq j$, $t_{i}\in [0,\tau_{i})$, $t_{j}\in [0,\tau_{j})$, and
\begin{align}\label{345}
T\bigl(t_{1},...,t_{n}\bigr)C^{n-1}=T_{1}\bigl(t_{1}\bigr) \cdot ... \cdot T_{n}\bigl(t_{n}\bigr),\quad \bigl(t_{1},...,t_{n}\bigr) \in \Omega.
\end{align}
Furthermore, the following holds:
\begin{itemize}
\item[(a)] If $C$ is injective, then $T_{i}(t_{i})T_{i}(t_{i}')=T_{i}(t_{i}')T_{i}(t_{i})$ for all $1\leq i\leq n$, $t_{i}\in [0,\tau_{i})$, $t_{i}'\in [0,\tau_{i})$ and $T(t)T(s)=T(s)T(t)$ for all $t,\ s\in \Omega.$
\item[(b)]
If $D=[0,+\infty)^{n},$ ${\mathcal A}=({\mathcal A}_{1},...,{\mathcal A}_{n})$ is the directional infinitesimal generator of $(T(t))_{t\in \Omega},$ $(T(t))_{t\in \Omega}$ is exponentially equicontinuous and the real numbers $\omega_{1},..., \omega_{n}$ satisfy that for each seminorm $p\in \circledast$ there exist a real number $M>0$ and a seminorm $q\in \circledast$ such that $p(T(t_{1}a_{1}+...+t_{n}a_{n})x)\leq M\exp(\omega_{1}t_{1}+...+\omega_{n}t_{n})q(x)$ for all $x\in X$ and $t=(t_{1},...,t_{n})\in [0,+\infty)^{n},$ then we have
\begin{align}\notag
C^{n-1} \int^{+\infty}_{0}...\int^{+\infty}_{0}& e^{-\lambda_{1}t_{1}-...-\lambda_{n}t_{n}}T\bigl(t_{1}a_{1}+...+t_{n}a_{n}\bigr)x\, dt_{1}\, ...\, dt_{n}
\\&  \label{345621}=\bigl( \lambda_{1}-{\mathcal A}_{1}\bigr)^{-1}C\cdot ... \cdot \bigl( \lambda_{n}-{\mathcal A}_{n}\bigr)^{-1}Cx,
\end{align}
for all $x\in X$ and $(\lambda_{1},...,\lambda_{n})\in {\mathbb C}^{n}$ such that $\Re \lambda_{i}>\omega_{i}$ for all $i\in {\mathbb N}_{n}.$
\end{itemize}
\item[(ii)] Suppose that 
$(T(t))_{t\in \Omega} $ 
is a $C$-semigroup and ${\mathcal A}=({\mathcal A}_{1},...,{\mathcal A}_{n})$ is the directional infinitesimal generator of $(T(t))_{t\in \Omega}$. Then for each $ x\in \bigcap_{i\in {\mathbb N}_{n}}D( {\mathcal A} _{i})$ the function $u(t):=T(t)x,$ $t\in \Omega $ is a   solution of the $n$-parameter abstract Cauchy inclusion $(ACP)_{\Omega}$.  The uniqueness of solution to $(ACP)_{\Omega}$ holds provided that the operator $C$ is injective, when the operator ${\mathcal A}_{i}$ is single-valued for $1\leq i \leq n.$
\end{itemize}
\end{thm}

\begin{proof}
Define $W(t_{1},...,t_{n}):=T(t_{1}a_{1}+...+t_{n}a_{n}),$ $t_{i}\in D_{i}$ ($1\leq i\leq n$); then it is clear that $(W(t))_{t\in D}$ is a $C$-semigroup, where $D:=D_{1}\times ... \times D_{n}.$ Using this observation and a simple argumentation, we may assume without loss of generality that $(a_{1},...,a_{n})=(e_{1},...,e_{n})$ and $D_{i}=[0,\tau_{i})$ for $1\leq i\leq n$.
 
Let $1\leq i, j\leq n$ and $i\neq j$. Then, due to (i.1), we have
\begin{align*}
T_{i}\bigl(t_{i}\bigr)T_{j}\bigl(t_{j}\bigr)=T\bigl(0,...,t_{i},...,t_{j},...,0\bigr)C=T_{j}\bigl(t_{j}\bigr)T_{i}\bigl(t_{i}\bigr).
\end{align*}
The functional equality \eqref{345} follows from (i.1) and a  simple computation. 

If the operator $C$ is injective, then \cite[Proposition 1.25]{knjiga} (see also the proof of  \cite[Lemma 2.1]{shaw2005}) implies $T_{i}(t_{i})T_{i}(t_{i}')=T_{i}(t_{i}')T_{i}(t_{i})$ for all $1\leq i\leq n$, $t_{i}\in [0,\tau_{i})$ and $t_{i}'\in [0,\tau_{i}).$ Let the points $t,\ s\in \Omega$ be fixed. Since $T(r)C=CT(r),$ $r\in \Omega,$ we have that  
$T(t)T(s)=T(s)T(t)$ if and only if $T(t)T(s)C^{2n-2}=T(s)T(t)C^{2n-2},$ i.e., if and only if
$$
\Bigl[T_{1}\bigl(t_{1}\bigr) \cdot ... \cdot T_{n}\bigl(t_{n}\bigr)\Bigr] \cdot \Bigl[ T_{1}\bigl(s_{1}\bigr) \cdot ... \cdot T_{n}\bigl(s_{n}\bigr)\Bigr]=\Bigl[ T_{1}\bigl(s_{1}\bigr) \cdot ... \cdot T_{n}\bigl(s_{n}\bigr)\Bigr] \cdot \Bigl[T_{1}\bigl(t_{1}\bigr) \cdot ... \cdot T_{n}\bigl(t_{n}\bigr)\Bigr].
$$
But, this equality simply follows from the equalities $T_{i}(t_{i})T_{j}(t_{j})=T_{j}(t_{j})T_{i}(t_{i})$ for all $1\leq i, j\leq n$ with $i\neq j$, $t_{i}\in [0,\tau_{i})$, $t_{j}\in [0,\tau_{j})$, and $T_{i}(t_{i})T_{i}(t_{i}')=T_{i}(t_{i}')T_{i}(t_{i})$ for all $1\leq i\leq n$, $t_{i}\in [0,\tau_{i})$ and $t_{i}'\in [0,\tau_{i}).$ This yields (a).

The  Laplace transform identity \eqref{345621} follows from \eqref{345}, the Fubini theorem and the equality
$$
\int^{+\infty}_{0}  e^{-\lambda_{i}t_{i}}T_{i}\bigl(t_{i}\bigr)x\, dt_{i}=\bigl( \lambda_{i}-{\mathcal A}_{i}\bigr)^{-1}Cx,\quad x\in X,\ \Re \lambda_{i}>\omega_{i} \ \ (1\leq i\leq n),
$$
which is a consequence of \cite[Theorem 3.2.5]{FKP}. This yields (b).

The first part of (ii) follows from a straightforward calculation and its proof is therefore omitted; the second part in (ii) is a very special consequence of \cite[Theorem 7.1.5]{apsclcs}.
\end{proof}

\section{Multiparameter $C$-cosine functions}\label{ccos}

We open this section by introducing the following notion:

\begin{defn}\label{1.2.1.1.1}
Let $0\in \Omega$ and $\Omega  \subseteq {\mathbb R}^{n}.$ A strongly continuous family
$(C(t))_{t\in \Omega}\subseteq L(X)$ 
is said to be a (local, if $\Omega$ is bounded) $C$-cosine function if the following conditions hold:
\begin{itemize}
\item[(i)] $C(t+s)C+C(t-s)C=2C(t)C(s)$ for all $t$, $s\in \Omega$ such that $t+s \in \Omega$ and $t-s\in \Omega$, and
\item[(ii)] $C(0)=C$.
\end{itemize}
If $\Omega=[0,+\infty)^{n},$ then it is said that $(C(t))_{t\in [0,+\infty)^{n}}$ is exponentially equicontinuous if there exist real numbers $\omega_{1},...,\omega_{n}$ such that for each seminorm $p\in \circledast$ there exist a real number $M>0$ and a seminorm $q\in \circledast$ such that $p(C(t)x)\leq M\exp(\omega_{1}t_{1}+...+\omega_{n}t_{n})q(x)$ for all $x\in X$ and $t=(t_{1},...,t_{n})\in [0,+\infty)^{n} $; if this is the case and $(X,\| \cdot \|)$ is a complex Banach space, then we also say that  $(C(t))_{t\in [0,+\infty)^{n}}$ is exponentially bounded.

Finally, if $C={\rm I}, $ then we also say that $(C(t))_{t\in \Omega}$ is a cosine function.
\end{defn}

If $n=1$ and $\Omega=[0,\tau)$ for some $0<\tau \leq +\infty,$ then a multivalued linear operator ${\mathcal A}$ which satisfies
\begin{itemize}
\item[(a)] $C(t){\mathcal A}\subseteq {\mathcal A}C(t)$, $t\in[0,\tau)$,
\item[(b)] $C(t)x-Cx=\int_0^t(t-s)C(s)y\,ds $,
$t\in[0,\tau)$, whenever $(x,y)\in {\mathcal A}$, 
\item[(c)] $(\int_0^t(t-s)C(s)x\,ds ,C(t)x-Cx )\in {\mathcal A}$ for all $x\in X$ and $t\in [0,\tau),$
\end{itemize}
is called a subgenerator of $(C(t))_{t\in[0,\tau)}$. In this case, the integral generator of $(C(t))_{t\in[0,\tau)}$ is defined by
\begin{align*}
&\hat{{\mathcal A}}:=\biggl\{(x,y)\in X\times X:C(t)x-Cx=\int_0^t(t-s)C(s)y\,ds,\;t\in[0,\tau)\biggr\},
\end{align*}
and it coincides with the infinitesimal generator of $(C(t))_{t\in[0,\tau)},$ which is defined by
\begin{align*}
{\mathcal A}_{inf} :=\Biggl\{ ( x,y) \in X \times X : 2\lim_{t\rightarrow 0}\frac{C(t)x-Cx}{t^{2}}=Cy\Biggr\}.
\end{align*}
We similarly define a subgenerator (the integral generator, the infinitesimal generator) of a $C$-cosine function $(C(t))_{t\in (-\tau,\tau)}$. Let us mention that
the conclusions from Remark \ref{konc}(ii) continue to hold for $C$-cosine functions.

Let $\tau_{i}\in (0,+\infty]$ for $1\leq i \leq n$, $\Omega =[0,\tau_{1}) \times ... \times [0,\tau_{n})$ or $\Omega =(-\tau_{1},\tau_{1}) \times ... \times (-\tau_{n},\tau_{n}),$ and let $(C(t))_{t\in \Omega}$ 
be a $C$-cosine function. Then it is easy to see that, for every $i\in {\mathbb N}_{n},$ the operator family $(C_{i}(t_{i})\equiv C(t_{i}e_{i}))_{0\leq t_{i} <\tau_{i}}$, resp. $(C_{i}(t_{i})\equiv C(t_{i}e_{i}))_{-\tau_{i}\leq t_{i} <\tau_{i}},$  is a local $C$-cosine function; furthermore, if 
$(C_{i}(t_{i}))_{0\leq t_{i} <\tau_{i}}$, resp. $(C_{i}(t_{i}))_{-\tau_{i}\leq t_{i} <\tau_{i}},$ is a local $C$-cosine function for some $i\in {\mathbb N}_{n}$, then it is easy to see that the expression 
$C(t_{1},...,t_{n}):=C_{i}(t_{i}),$ $t=(t_{1},...,t_{n})\in \Omega\equiv {\mathbb R} \times ...\times {\mathbb R} \times [0,\tau_{i}) \times ... \times {\mathbb R},$ where $\Omega = {\mathbb R} \times ...\times {\mathbb R} \times [0,\tau_{i}) \times ... \times {\mathbb R}$ or $\Omega ={\mathbb R} \times ...\times {\mathbb R} \times (-\tau_{i},\tau_{i}) \times ... \times {\mathbb R}$ determines a $C$-cosine function. 

\begin{example}\label{pas}
Let $(C(t))_{t\in {\mathbb R}}$ be a $C$-cosine function, and let $(a_{1},...,a_{n})\in {\mathbb R}^{n}.$ Define
$$
C_{a}\bigl(t_{1},...,t_{n}\bigr):=C\bigl( a_{1}t_{1}+...+a_{n}t_{n}\bigr),\quad t=\bigl(t_{1},...,t_{n}\bigr) \in {\mathbb R}^{n}.
$$
Then $(C_{a}(t))_{t\in {\mathbb R}^{n}}$ is a $C$-cosine function, and $(C_{a;i}(t_{i})\equiv C_{a}(t_{i}e_{i})=C(a_{i}t_{i}))_{t_{i} \in {\mathbb R}}$ is subgenerated by $a_{i}{\mathcal A}$ if $(C(t))_{t \in {\mathbb R}}$  is subgenerated by ${\mathcal A}$ ($1\leq i\leq n$). 
\end{example}

Our first structural result for multiparameter $C$-cosine functions reads as follows:

\begin{prop}\label{rambo}
Suppose that $0\in \Omega$, $\Omega  \subseteq {\mathbb R}^{n}$, $\Omega=-\Omega$  is convex, and $\Omega$ has a nonempty interior. If $(C(t))_{t\in \Omega}$ 
is a $C$-cosine function and the operator $C$ is injective, then $C(t)C(s)=C(s)C(t)$ for all $t,\ s\in \Omega.$
\end{prop}

\begin{proof}
Suppose first that $t,\ s,\ t\pm s\in \Omega.$ Since $\Omega=-\Omega$, we have $s-t\in \Omega$ and the functional equality of $C$-cosine functions
gives 
\begin{align}\label{vera}
C(t+s)C+C(t-s)C=2C(t)C(s)\mbox{ and }C(t+s)C+C(s-t)C=2C(s)C(t).
\end{align}
On the other hand, the segment $[s-t,t-s]$ belongs to $\Omega$ since $[s-t,t-s]=[s-t,0] \cup [0,t-s]$ and $\Omega$ is a convex set. Since $C$ is injective and the operator family $(C(r(t-s)))_{r\in [-1,1]}$ is a local $C$-cosine function, it follows that $C((-1)(t-s))=C(t-s),$ i.e., $C(s-t)=C(t-s);$ see, e.g., \cite{tajw}. Keeping in mind \eqref{vera}, it follows that $C(t)C(s)=C(s)C(t)$.

Suppose now that $t,\ s\in \Omega.$ Then there exist two possibilites: $t\in \Omega^{\circ}$ or $t\in \partial \Omega.$ In the first case, we have $s/2=(1/2) \cdot 0+(1/2) \cdot s \in \Omega$ by the convexity of $\Omega$; inductively, we get $s/2^{k}\in  \Omega$ for all $k\in {\mathbb N}$, and we can write
\begin{align*}
C(s)C=2C(s/2)^{2}-C^{2}; \ \ C(s)C^{3}=2\Bigl[ 2C(s/4)^{2}-C^{2}\Bigr]^{2}-C^{5}=....
\end{align*}
Therefore, for every $k\in {\mathbb N}$, there exist an integer $c_{k}\in {\mathbb N}$ and a real polynomial $p_{k}(\cdot)$ such that
\begin{align}\label{ram}
C(s)C^{c_{k}}=p_{k}\Bigl(C\bigl(s/2^{k}\bigr)\Bigr),\quad k\in {\mathbb N}.
\end{align}
By the foregoing, we have $C(t)C(s/2^{k})=C(s/2^{k})C(t)$ for a sufficiently large integer $k\in {\mathbb N}$, so that \eqref{ram} together with the injectivity of $C$ implies $C(t)C(s)=C(s)C(t)$,
as claimed.
If $t\in \partial \Omega,$ then the equality $C(t)C(s)=C(s)C(t)$ follows from the previously proved case combined with the strong continuity of $(C(t))_{t\in \Omega} $ and the fact that $\Omega^{\circ}$ is dense in $\overline{\Omega},$ which follows from our assumptions that $\Omega  $ is convex and $\Omega^{\circ} \neq \emptyset$.  
\end{proof}

In the one-dimensional setting, in the definition of a $C$-cosine function $(C(t))_{t\in [0,\tau)},$ where $0<\tau \leq +\infty,$ it is assumed that $C(t+s)C+C(|t-s|)C=2C(t)C(s)$, provided that $0\leq t,\ s,\ t+s<\tau;$ see \cite[p. 5]{tajw}. This immediately implies that  $C(t)C(s)=C(s)C(t)$ for $0\leq t,\ s,\ t+s<\tau .$ In the multiparameter case $n\geq 2,$ we cannot define the term $C(|t-s|)$ and the equality $C(t)C(s)=C(s)C(t)$ does not follow directly for $t,\ s,\ t+s\in \Omega $. In connection with this observation, we would like to state  the following results, which can be deduced with the help of argumentation contained in the proof of Proposition \ref{rambo}:

\begin{prop}\label{df}
\begin{itemize}
\item[(i)] 
Suppose that $0\in \Omega$, $\Omega  \subseteq {\mathbb R}^{n}$ is convex, $\Omega$ has a nonempty interior, and $(C(t))_{t\in \Omega}$ 
is a $C$-cosine function. If the operator $C$ is injective and $C(t)C(s)=C(s)C(t)$ for any $t,\ s\in \Omega$ such that $t+s \in \Omega$, then $C(t)C(s)=C(s)C(t)$ for any $t,\ s\in \Omega$.
\item[(ii)] Suppose that $0\in \Omega$, $\Omega  \subseteq {\mathbb R}^{n}$ is convex, the operator $C$ is injective, $(C(t))_{t\in \Omega}$ 
is a $C$-cosine function and $C(t)C(s)=C(s)C(t)$ for any $t,\ s\in \Omega$. Define $\Omega':=\Omega \cup (-\Omega)$ and $\overline{C} : \Omega' \rightarrow L(X)$ by $\overline{C}(t):=C(t)$ for all $t\in \Omega$, and $\overline{C}:=C(-t)$ for all $t\in -\Omega \setminus \Omega.$ Then $(\overline{C}(t))_{t\in \Omega'}$ is a $C$-cosine function.
\end{itemize}
\end{prop}

Suppose now that $(X,\| \cdot \|)$ is a complex Banach space, $\Omega=[0,+\infty)^{n}$ and $(T(t))_{t\in [0,+\infty)^{n}}$ is a strongly continuous semigroup. Then the equality \eqref{345} shows that $(T(t))_{t\in [0,+\infty)^{n}}$ is exponentially bounded.
In the subsequent result, we will prove the same 
result for cosine functions and therefore 
extend the statement of \cite[Lemma 3.14.3; a)]{a43} to the higher-dimensional setting:

\begin{prop}\label{cos}
If $(X,\| \cdot \|)$ is a complex Banach space and $(C(t))_{t\in [0,+\infty)^{n}}$ is a cosine function, then $(C(t))_{t\in [0,+\infty)^{n}}$ is exponentially bounded.
\end{prop}

\begin{proof}
We will prove the statement in the two-dimensional setting; the general case can be covered similarly. Put $M:=\sup_{t\in [0,2]^{2}}\| C(t)\|<+\infty$ and assume that the real numbers $\omega_{i}\geq 0$ satisfy $2\| C(e_{i})\| e^{-\omega_{i}}+e^{-2\omega_{i}}\leq 1$ for $i=1,2.$ We claim that $\| C(t_{1},t_{2})\| \leq M\exp(\omega_{1}t_{1}+\omega_{2}t_{2})$ for all $(t_{1},t_{2}) \in [0,+\infty)^{2}.$  This is clearly true if $(t_{1},t_{2}) \in [0,2]^{2};$ let us assume that this is true for any point $(t_{1},t_{2}) \in [0,k]^{2}$, where $k\in {\mathbb N}$ and $k\geq 2,$ ans let us prove that this is also true for any point $(t_{1},t_{2}) \in [0,k+1]^{2}.$ If  $(t_{1},t_{2}) \in [1,k] \times [0,k],$ then the induction hypothesis implies:
\begin{align*}&
\bigl\| C(t+e_{1})\bigr\| =\bigl\| 2C(t)C(e_{1})-C(t-e_{1})\bigr\| \leq  2M\bigl\| C(e_{1})\bigr\|e^{\omega_{1}t_{1}+\omega_{2}t_{2}}+M\bigl\| C(t-e_{1}) \bigr\| 
\\& \leq 2M\bigl\| C(e_{1})\bigr\|e^{\omega_{1}t_{1}+\omega_{2}t_{2}}+Me^{\omega_{1}(t_{1}-1)+\omega_{2}t_{2}}
\\& \leq Me^{\omega_{1}(t_{1}+1)+\omega_{2}t_{2}}\Bigl[ 2\| C(e_{1})\| e^{-\omega_{1}}+e^{-2\omega_{1}} \Bigr] \leq Me^{\omega_{1}(t_{1}+1)+\omega_{2}t_{2}}.
\end{align*} 
The above implies that, for every $ (t_{1},t_{2}) \in [0,k+1] \times [0,k],$ we have:
\begin{align*}&
\bigl\| C(t+e_{2})\bigr\| =\bigl\| 2C(t)C(e_{2})-C(t-e_{2})\bigr\| \leq  2M\bigl\| C(e_{2})\bigr\|e^{\omega_{1}t_{1}+\omega_{2}t_{2}}+M\bigl\| C(t-e_{2}) \bigr\| 
\\& \leq 2M\bigl\| C(e_{2})\bigr\|e^{\omega_{1}t_{1}+\omega_{2}t_{2}}+Me^{\omega_{1}t_{1}+\omega_{2}(t_{2}-1)}
\\& \leq Me^{\omega_{1}t_{1}+\omega_{2}(t_{2}+1)}\Bigl[2\| C(e_{2})\| e^{-\omega_{2}}+e^{-2\omega_{2}} \Bigr] \leq Me^{\omega_{1}t_{1}+\omega_{2}(t_{2}+1)},
\end{align*} 
which completes the proof of result.
\end{proof}

\begin{rem}\label{renik}
\begin{itemize}
\item[(i)] Suppose that $(X,\| \cdot \|)$ is a complex Banach space, $(a_{1},...,a_{n})$ is a normed basis of ${\mathbb R}^{n}$, $D_{i}=[0,+\infty)$ for all $i\in {\mathbb N}_{n},$ and $(T(t))_{t\in \Omega}$ is a strongly continuous semigroup of operators, where  $\Omega$ is given by \eqref{me1}.   
In \cite[Proposition 9.1.1]{apsclcs}, we have proved that there exist real numbers $M\geq 1$ and $\omega_{i}\in {\mathbb R}$ ($1\leq i \leq n$) such that, for every $(t_{1},...,t_{n})\in [0,+\infty)^{n},$ we have
\begin{align}\label{maww}
\bigl \| T(t)\bigr\| \leq Me^{\omega_{1}t_{1}+...+\omega_{n}t_{n}},\mbox{ provided that } t=t_{1}a_{1}+...+t_{n}a_{n} \in \Omega.
\end{align}
Using a similar argumentation as in the proof of above-mentioned result, we can slightly extend Proposition \ref{cos} and prove that the estimate \eqref{maww} holds if $(T(t))_{t\in \Omega}$ is replaced by a cosine function $(C(t))_{t\in \Omega}$.
\item[(ii)] If $C\neq {\rm I} $ and $\Omega=[0,+\infty)$, then it is well-known that there are many examples of not exponentially bounded $C$-semigroups and $C$-cosine functions; see, e.g., \cite{knjiga}. 
\end{itemize}
\end{rem}

In connection with Proposition \ref{cos}, the following question arises immediately: If $\Omega=[0,+\infty)^{n}$ and $(C(t))_{t\in [0,+\infty)^{n}}$ is exponentially equicontinuous, how we can compute its Laplace transform, i.e., what is the value of term
\begin{align*}
\int^{+\infty}_{0}...\int^{+\infty}_{0}e^{-\lambda_{1}t_{1}-...-\lambda_{n}t_{n}}C\bigl(t_{1},...,t_{n}\bigr)x\, dt_{1}\, ...\, dt_{n},
\end{align*}
where $x\in X$ is given? The result is very simple in the case that $(C(t))_{t\in [0,+\infty)^{n}}$ is given by \eqref{3456}, when the class of multiparameter $C$-cosine functions is a very special subclass of the class of regularized resolvent families introduced in \cite[Definition 2.1(iii)]{aims}. 
In the case that $n\geq 2$ and $(C(t))_{t\in [0,+\infty)^{n}}$ is not given by \eqref{3456}, the method proposed in the proof of \cite[Proposition 3.14.4]{a43} is not applicable and the best we can do is to prove the following unsatisfactory result:

\begin{prop}\label{asd}
Suppose that $\Omega={\mathbb R}^{n}$, $(C(t))_{t\in {\mathbb R}^{n}}$ is a $C$-cosine function and there exist real numbers $\omega_{1},...,\omega_{n}$ such that for each seminorm $p\in \circledast$ there exist a real number $M>0$ and a seminorm $q\in \circledast$ such that $p(C(t)x)\leq M\exp(\omega_{1}|t_{1}|+...+\omega_{n}|t_{n}|)q(x)$ for all $x\in X$ and $t=(t_{1},...,t_{n})\in {\mathbb R}^{n}.$ Put 
$$
R(\lambda ;s)x:=\int_{s+[0,+\infty)^{n}}e^{-\lambda_{1}t_{1}-...-\lambda_{n}t_{n}}C\bigl(t_{1},...,t_{n}\bigr)x\, dt_{1}\, ...\, dt_{n},
$$
for $ s\in {\mathbb R}^{n},\ \lambda =(\lambda_{1},...,\lambda_{n}) \in {\mathbb C}^{n},\ x\in X$ and $\Re \lambda_{i} >\omega_{i}\  (1\leq i\leq n).$
Then we have
\begin{align}\notag
2R(\lambda;0)R(\nu;0) x &=\int_{[0,+\infty)^{n}}e^{-(\lambda_{1}+\nu_{1})s_{1}-...-(\lambda_{n}+\nu_{n})s_{n}}R(\lambda;s)Cx\, ds_{1}\, ...\, ds_{n} 
\\& \label{klasa}+\int_{[0,+\infty)^{n}}e^{-(\nu_{1}-\lambda_{1})s_{1}-...-(\nu_{n}-\lambda_{n})s_{n}}R(\lambda;-s)Cx\, ds_{1}\, ...\, ds_{n},
\end{align}
provided that $t=(t_{1},...,t_{n}),\ s=(s_{1},...,s_{n})\in {\mathbb R}^{n},\ \lambda =(\lambda_{1},...,\lambda_{n}) \in {\mathbb C}^{n},$ $\nu =(\nu_{1},...,\nu_{n}) \in {\mathbb C}^{n},$ $x\in X$ and $\min(\Re \lambda_{i},\Re \nu_{i}) >\omega_{i}\  (1\leq i\leq n).$
\end{prop}

\begin{proof}
Let $x\in X$ and let the prescribed assumptions hold.
Integrating the both sides of functional equality $C(t+s)C+C(t-s)C=2C(t)C(s)$, $t,\ s\in {\mathbb R}^{n},$ we obtain with the help of Fubini theorem that:
\begin{align*}
2R(\lambda;0)R(\nu;0)x &=\int_{[0,+\infty)^{n}}\int_{[0,+\infty)^{n}}e^{-\lambda t}e^{-\nu s}\bigl[ C(t+s)Cx+C(t-s)Cx \bigr]\, dt\, ds
\\& = \int_{[0,+\infty)^{n}}e^{-\nu s}\Biggl[ \int_{s+[0,+\infty)^{n}} e^{-\lambda (v-s)}C(v)Cx\, dv \Biggr]\, ds
\\& +\int_{[0,+\infty)^{n}}e^{-\nu s}\Biggl[ \int_{-s+[0,+\infty)^{n}} e^{-\lambda (v+s)}C(v)Cx\, dv \Biggr]\, ds
\\& = \int_{[0,+\infty)^{n}}e^{-(\lambda +\nu) s}R(\lambda;s)Cx\, ds + \int_{[0,+\infty)^{n}}e^{-(\nu-\lambda) s}R(\lambda;-s)Cx\, ds.
\end{align*}
This implies \eqref{klasa} and completes the proof of proposition; here, we set $\lambda t:=\lambda_{1}t_{1}+...+\lambda_{n}t_{n}$ and $dt:=dt_{1}\, ....\, dt_{n}.$
\end{proof}

\subsection{Abstract second-order Cauchy problems}\label{gsd}

The basic relation between multiparameter $C$-semigroups and multiparameter $C$-cosine functions is given in the following result (let us recall that a complex Banach space $X$ has
the UMD property if the Hilbert
transform is bounded in $L^{p}({\mathbb R}: X);$ if $\Omega={\mathbb R}$, then every bounded
cosine function $(C(t))_{t\in {\mathbb R}}$ in a UMD-space $X$ can be represented by \eqref{3456}, see \cite[Section 3.16]{a43}):

\begin{thm}\label{emoj}
Suppose that $0\in \Omega$, $\Omega  \subseteq {\mathbb R}^{n} $ and $\Omega=-\Omega.$ If $(T(t))_{t\in \Omega}$ is a $C$-semigroup, then we set
\begin{align}\label{3456}
C(t):=\frac{1}{2}\Bigl[ T(t)+T(-t)\Bigr],\quad t\in \Omega .
\end{align}
Then $(C(t))_{t\in \Omega}$ is a $C$-cosine function. 

Suppose now that $\tau_{i}\in (0,+\infty]$ for $1\leq i \leq n$, $\Omega =(-\tau_{1},\tau_{1}) \times ... \times (-\tau_{n},\tau_{n})$, ${\mathcal A}_{i}$ is a subgenerator of $(T_{i}(t_{i}) )_{0\leq \tau <\tau_{i}}$ 
and
${\mathcal A}_{i}^{2}$ is closed ($i\in  {\mathbb N}_{n}$). Then the following holds:
\begin{itemize}
\item[(i)] The operator
${\mathcal A}_{i}^{2}$ is
a subgenerator of
$(C_{i}(t_{i}) )_{0\leq t_{i} <\tau_{i}}$ ($i\in  {\mathbb N}_{n}$).
\item[(ii)] For every $x\in \bigcap_{1\leq i \leq n}D({\mathcal A}_{i}^{2}),$ the function $u(t):=C(t)C^{n-1}x,$ $t\in \Omega$ is a solution of the $n$-parameter Cauchy inclusion of second order $(ACP_{2}).$ 
\item[(iii)] The uniqueness of solutions to $(ACP_{2})$ holds if the operator $C$ is injective.
\end{itemize}
\end{thm}

\begin{proof}
It is clear that the operator family $(C(t))_{t\in \Omega}$ is strongly continuous and $C(0)=C.$ 
Suppose now that $t,\ s,\ t\pm s\in \Omega ;$ since $\Omega=-\Omega,$ we have $-t-s\in \Omega$ and $s-t\in \Omega.$ Then the first part of proposition follows from the next simple computation:
\begin{align*}
C(t+s)C+C(t-s)C&=\frac{1}{2}\Bigl[ T(t+s)+T(-t-s)\Bigr]+\frac{1}{2}\Bigl[ T(t-s)+T(s-t)\Bigr]
\\& =2\cdot \frac{1}{2}\Bigl[ T(t)+T(-t)\Bigr] \cdot \frac{1}{2}\Bigl[ T(s)+T(-s)\Bigr].
\end{align*}
The issue (i) follows similarly as in the proof of \cite[Proposition 2.1.17]{knjigah} and therefore omitted. 

To prove (ii), fix an element $x\in \bigcap_{1\leq i \leq n}D({\mathcal A}_{i}^{2})$. Then, clearly, $u(0)=C^{n}x$ and there exist elements $y\in X$ and $z\in X$ such that $(x,y)\in {\mathcal A}_{1}$ and  $(y,z)\in {\mathcal A}_{1}.$ Due to \eqref{345}, we have $u(t)=1/2[   
\prod_{1\leq j\le n} T_{j}( t_{j})x+\prod_{1\leq j\le n} T_{j}( -t_{j})x],$ $t\in \Omega$ and $T_{i}(t_{i})T_{j}(t_{j})=T_{j}(t_{j})T_{i}(t_{i})
 $ for $t_{i} \in (-\tau_{i},\tau_{i}),$ $ t_{j} \in (-\tau_{j},\tau_{j}),$ $i\neq j.$ Since $T_{1}(t_{1})x-Cx=\int^{t_{1}}_{0}T_{1}(s_{1})y\, ds_{1},$ $t_{1}\in [0,\tau_{1}),$ the above implies $T_{1}(t_{1})T_{2}(t_{2})\cdot ... \cdot T_{n}(t_{n})x-CT_{2}(t_{2})\cdot ... \cdot T_{n}(t_{n})x=\int^{t_{1}}_{0}T_{1}(s_{1})T_{2}(t_{2})\cdot ... \cdot T_{n}(t_{n})y\, ds_{1},$ $t_{1}\in [0,\tau_{1})$ and
\begin{align}\label{prc321}
\frac{\partial u}{\partial t_{i}}(t)=\frac{1}{2}\Bigl[ T_{1}(t_{1})T_{2}(t_{2})\cdot ... \cdot T_{n}(t_{n})y-T_{1}(-t_{1})T_{2}(-t_{2})\cdot ... \cdot T_{n}(-t_{n})y\Bigr],\quad t\in \Omega.
\end{align}
Therefore,
$$
\Bigl( \frac{\partial u}{\partial t_{1}}\Bigr)_{t=(0,t_{2},...,t_{n})} \in \frac{1}{2}{\mathcal A}_{1}\Biggl[       
\prod_{2\le j\le n} T_{j}\bigl( t_{j}\bigr)x-\prod_{2\leq j\leq n\atop} T_{j}\bigl( -t_{j}\bigr)x\Biggr],\quad t\in \Omega.
$$
Similarly we obtain that for each $i\in {\mathbb N}_{n} \setminus \{1\}$ and $t\in \Omega$ we have $$
\Bigl( \frac{\partial u}{\partial t_{i}}\Bigr)_{t=(t_{1},...,t_{i-1},0,t_{i+1},...,t_{n})} \in \frac{1}{2}{\mathcal A}_{i}\Biggl[       
\prod_{1\le j\le n\atop i\ne j} T_{j}\bigl( t_{j}\bigr)x-\prod_{1\leq j\leq n\atop i\neq j} T_{j}\bigl( -t_{j}\bigr)x\Biggr].$$
Differentiating \eqref{prc321} once more, we get $u\in C^{2}(\Omega : X)$,
$$
\frac{\partial ^{2}u}{\partial t_{i}^{2}}(t)=\frac{1}{2}\Bigl[ T_{1}(t_{1})T_{2}(t_{2})\cdot ... \cdot T_{n}(t_{n})z+T_{1}(-t_{1})T_{2}(-t_{2})\cdot ... \cdot T_{n}(-t_{n})z\Bigr],\quad t\in \Omega ,
$$
and therefore $(\partial ^{2}u/\partial t_{i}^{2})(t)\in {\mathcal A}_{i}^{2}u(t),\;t \in \Omega,\ 1\leq i\leq n,$ so that $u(\cdot)$ is  a solution of problem $(ACP_{2}).$ 

Finally, if the operator $C$ is injective, then ${\mathcal A}_{i}=A_{i}$ is a closed linear operator for $1\leq i \leq n.$ If $u_{1}(\cdot)$ and $u_{2}(\cdot)$ are solutions to $(ACP_{2}),$ then we define $u:=u_{1}-u_{2}.$ Then it is clear that  $
u\in C^{2}(\Omega : X),$ $(\partial ^{2}u/\partial t_{i}^{2})(t)\in {\mathcal A}_{i}^{2}u(t),\;t \in \Omega,\ 1\leq i\leq n,$ $
u(0)=0$ and $( \frac{\partial u}{\partial t_{i}})_{t=(t_{1},...,t_{i-1},0,t_{i+1},...,t_{n})} =0$ for all $ t\in \Omega$ and $ 1\leq i\leq n.$ Since condition (D1) holds with the operator $ {\mathcal A}_{i}$ replaced therein with the operator $ {\mathcal A}_{i}^{2}$ and the number $\alpha_{i}=2$ ($1\leq i\leq n$), we can argue similarly as in the proof of \cite[Theorem 8.1.23]{apsclcs} to obtain $u\equiv 0. $ 
\end{proof}

Concerning the formulae \eqref{cc} and \eqref{3456}, we have the following observations:

\begin{rem}\label{pf}
Suppose that $(X,\| \cdot \|)$ is a complex Banach space, $A_{i}\in L(X)$ for $1\leq i\leq n$ and $A_{i}A_{j}=A_{j}A_{i}$ for $1\leq i,\ j\leq n.$ 
\begin{itemize}
\item[(i)] Then $T(t_{1},...,t_{n})=\exp(t_{1}A_{1}+...+t_{n}A_{n}),$ $t=(t_{1},...,t_{n})\in {\mathbb R}^{n};$ applying the polynomial formula, we get that the $C$-cosine function $(C(t))_{t\in {\mathbb R}^{n}}$ given by \eqref{3456} can be represented in the following way:
\begin{align*}
C(t)&=\frac{1}{2}\Bigl[ e^{t_{1}A_{1}+...+t_{n}A_{n}}+e^{-t_{1}A_{1}-...-t_{n}A_{n}}\Bigr]
\\& =\frac{1}{2}\sum_{k=0}^{+\infty}\frac{\bigl( t_{1}A_{1}+...+t_{n}A_{n}\bigr)^{k}+\bigl( -t_{1}A_{1}-...-t_{n}A_{n}\bigr)^{k}}{k!}
\\& =\frac{1}{2}\sum_{k=0}^{+\infty}\sum_{m_{1}+...+m_{n}=k \atop (m_{1},...,m_{n})\in {\mathbb N}_{0}^{n}}\frac{t_{1}^{m_{1}}A_{1}^{m_{1}}\cdot ... \cdot t_{n}^{m_{n}}A_{n}^{m_{n}}+ t_{1}^{m_{1}}\bigl(-A_{1}\bigr)^{m_{1}}\cdot ... \cdot t_{n}^{m_{n}}\bigl(-A_{n}\bigr)^{m_{n}}}{m_{1}! \cdot ... \cdot m_{n}!},
\end{align*}
for any $t=(t_{1},...,t_{n})\in {\mathbb R}^{n}.$
\item[(ii)] In this part, we will use the notation $A^{m}:= A_{1}^{m_{1}}\cdot ... \cdot A_{n}^{m_{n}}$ for $m=(m_{1},..., m_{n})\in {\mathbb N}_{0}^{n}.$ Suppose that 
\begin{align}\label{for}
C(t):=\sum_{k\in {\mathbb N}_{0}^{n}}P_{k}(t)A^{k},\quad t\in {\mathbb R}^{n},
\end{align} 
where $P_{k}(\cdot)$ is an even complex polynomial depending on $n$ real variables ($k\in {\mathbb N}_{0}^{n}$) and the series in \eqref{for} absolutely converges in the strong operator topology for all $t\in {\mathbb R}^{n}.$ Then
$$
C(t)x+C(s)x=\sum_{k\in {\mathbb N}_{0}^{n}}\Bigl[ P_{k}(t)+P_{k}(s)\Bigr]A^{k},\quad t\in {\mathbb R}^{n},\ x\in X,
$$ 
and
\begin{align*}&
2C(t)C(s)x=2C(t)\sum_{m\in {\mathbb N}_{0}^{n}}P_{m}(s)A^{m}x=2\sum_{l\in {\mathbb N}_{0}^{n}}P_{l}(t)A^{l}\sum_{m\in {\mathbb N}_{0}^{n}}P_{m}(s)A^{m}x
\\& =2\sum_{l\in {\mathbb N}_{0}^{n}}\sum_{m\in {\mathbb N}_{0}^{n}}P_{l}(t) P_{m}(s)A^{m+l}x
\\&=2\sum_{k\in {\mathbb N}_{0}^{n}}\Biggl[ \sum_{l,  m\in {\mathbb N}_{0}^{n}; l+m=k}P_{l}(t) P_{m}(s)\Biggr]A^{k}x,\quad t,\ s\in {\mathbb R}^{n},\ x\in X.
\end{align*}
Therefore, the operator family $(C(t))_{t\in {\mathbb R}^{n}},$ given by \eqref{for}, will be a cosine function if the set of polynomials $(P_{k}(\cdot))_{k\in {\mathbb N}_{0}^{n}}$ satisfies, besides the standard assumptions, that
\begin{align}\label{for1}
P_{k}(t)+P_{k}(s)=2\sum_{l,  m\in {\mathbb N}_{0}^{n}; l+m=k}P_{l}(t) P_{m}(s),\quad t,\ s\in {\mathbb R}^{n}\ \ \bigl(k\in {\mathbb N}_{0}^{n}\bigr).
\end{align} 
Clearly, \eqref{for1} implies 
\begin{align}\label{for2}
P_{0}(t)+P_{0}(s)=2P_{0}(t) P_{0}(s),\quad t,\ s\in {\mathbb R}^{n},
\end{align}
and{\small
\begin{align}\label{for3}
P_{k}(t)\Bigl[1-2P_{0}(s)\Bigr]+P_{k}(s)\Bigl[1-2P_{0}(t)\Bigr]=2\sum_{l,  m\in {\mathbb N}_{0}^{n}; l+m=k \atop |l|<k,\ |m|<k}P_{l}(t) P_{m}(s),\ t,\ s\in {\mathbb R}^{n}\  \bigl(k\in {\mathbb N}_{0}^{n}\bigr).
\end{align}}By \eqref{for2}, it readily follows that $P_{0}(\cdot)$ has to be a constant polynomial; since $C(0)={\rm I},$ we have $P_{0}(0)=1$ and therefore $P_{0}\equiv 1.$ Inserting this in \eqref{for3}, we get
\begin{align}\label{for4}
-P_{k}(t)-P_{k}(s)=2\sum_{l,  m\in {\mathbb N}_{0}^{n}; l+m=k \atop |l|<k,\ |m|<k}P_{l}(t) P_{m}(s),\quad t,\ s\in {\mathbb R}^{n}\ \ \bigl(k\in {\mathbb N}_{0}^{n} \setminus \{0\}\bigr).
\end{align}
Further analysis of polynomial equation \eqref{for4} is without scope of this paper (of course, \eqref{for4} has infinitely many solutions; see, e.g., Example \ref{pas} with $X:={\mathbb C}$, $n=2,$ $(A_{1},A_{2}):=({\rm I}, {\rm I})$ and the cosine function $(C(t_{1},t_{2}))_{(t_{1},t_{2})\in {\mathbb R}^{2}}$ given by $C(t_{1},t_{2}):=\cos (t_{1}),$ $C(t_{1},t_{2}):=\cos (t_{2})$  or $C(t_{1},t_{2}):=\cos (at_{1}+bt_{2}),$ $(t_{1},t_{2})\in {\mathbb R}^{2}$, where $a,\ b\in {\mathbb R}$).
\end{itemize}
\end{rem}

We continue by stating the following result:

\begin{thm}\label{novio} 
Let $\tau_{i}\in (0,+\infty]$ for $1\leq i \leq n$, let $\Omega =(-\tau_{1},\tau_{1}) \times ... \times (-\tau_{n},\tau_{n}),$ and let $(C(t))_{t\in \Omega}$ 
be a $C$-cosine function such that $C(t)C(s)=C(s)C(t)$ for all $t,\ s\in \Omega$. If ${\mathcal A}_{i}$ is the integral generator of 
$(C_{i}(t_{i}))_{-\tau_{i}< t_{i} <\tau_{i}}$ for $1\leq i\leq n,$ then the following holds:
\begin{itemize}
\item[(i)] If $  i\in {\mathbb N}_{n}$ is fixed and $x\in D({\mathcal A}_{i}),$ then the function 
\begin{align}\label{dfg}
u_{i}(t):=C\bigl(t_{1},...,t_{i-1},t_{i},t_{i+1},...,t_{n} \bigr)Cx+C\bigl(-t_{1},...,-t_{i-1},t_{i},-t_{i+1},...,-t_{n} \bigr)Cx,\ t\in \Omega
\end{align}
 is a solution of the $n$-parameter Cauchy inclusion of second order $(ACP_{2;i}).$
\item[(ii)] If $n=2$ and $x\in D({\mathcal A}_{1}) \cap D({\mathcal A}_{2}),$ then $u_{1}=u_{2}=:u$ is a solution of the two-parameter Cauchy inclusion of second order $(ACP_{2})'.$ 
Furthermore, if the operator $C$ is injective, then ${\mathcal A}_{1}$ and ${\mathcal A}_{2}$ are closed linear operators and $u(t_{1},t_{2}):=C(t_{1},t_{2})Cx+C(t_{1},-t_{2})Cx,$ $t=(t_{1},t_{2})\in \Omega$ is a unique solution of problem $(ACP_{2})'.$
\item[(iii)] If $n \geq 2$ and $x\in D({\mathcal A}_{1}) \cap D({\mathcal A}_{2}) \cap ... \cap D({\mathcal A}_{n})$, then the function
\begin{align}\label{sigme}
u\bigl(t_{1},...,t_{n}\bigr):=\sum_{(\sigma_{1},...,\sigma_{n}) \in \{ -1,1\}^{n}}C\bigl( \sigma_{1}t_{1},..., \sigma_{n}t_{n}\bigr)Cx,\quad t=\bigl(t_{1},...,t_{n}\bigr) \in \Omega,
\end{align}
is a solution of problem $(ACP_{2})''.$
Furthermore, if the operator $C$ is injective, then ${\mathcal A}_{i}$ is a closed linear operator for each $i\in {\mathbb N}_{n},$ and the function $u(\cdot)$, given by \eqref{sigme}, is a unique solution of problem $(ACP_{2})''.$
\end{itemize}
\end{thm}

\begin{proof}
It is clear that $u_{i}(t_{1},...,t_{i-1},0,t_{i+1},...,t_{n})=C(t_{1},...,t_{i-1},0,t_{i+1},...,t_{n})Cx+C(-t_{1},...,-t_{i-1},0,-t_{i+1},...,-t_{n})Cx,$ if $t= (t_{1},...,t_{i-1},0,t_{i+1},...,t_{n} ) \in \Omega.$ We will prove the remainder of part (i) in the case that $i=1;$ the general case follows in the same way. Let $y\in {\mathcal A}_{1}x.$  Since we have assumed that $C(t)C(s)=C(s)C(t)$ for all $t,\ s\in \Omega$, it simply follows that $C(t)y\in {\mathcal A}_{1}C(t)x$
for all $t\in \Omega.$ Now we multiply the both sides of equality $C(t_{1},0,...,0)x-Cx=\int^{t_{1}}_{0}(t_{1}-s_{1})C(s_{1},0,...,0)y\, ds_{1},$ $t_{1}\in (-\tau_{1},\tau_{1})$ with $2C(0,t_{2},...,t_{n})$ and use the functional equality of multiparameter $C$-cosine functions to obtain that
\begin{align*}
C\bigl(t_{1},t_{2},...,t_{n}\bigr)Cx&-2C\bigl(0,t_{2},...,t_{n}\bigr)Cx+C\bigl(t_{1},-t_{2},...,-t_{n}\bigr)Cx
\\&=2C\bigl(0,t_{2},...,t_{n}\bigr)\int^{t_{1}}_{0}\bigl(t_{1}-s_{1}\bigr)C\bigl(s_{1},0,...,0\bigr)y\, ds_{1};
\end{align*}  
hence,
\begin{align*}&
C\bigl(t_{1},t_{2},...,t_{n}\bigr)Cx+C\bigl(t_{1},-t_{2},...,-t_{n}\bigr)Cx
\\& =2C\bigl(0,t_{2},...,t_{n}\bigr)Cx+2C\bigl(0,t_{2},...,t_{n}\bigr)\int^{t_{1}}_{0}\bigl(t_{1}-s_{1}\bigr)C\bigl(s_{1},0,...,0\bigr)y\, ds_{1}
\end{align*}  
and

\begin{align*}&
C\bigl(t_{1}+h_{1},t_{2},...,t_{n}\bigr)Cx+C\bigl(t_{1}+h_{1},-t_{2},...,-t_{n}\bigr)Cx
\\& =2C\bigl(0,t_{2},...,t_{n}\bigr)Cx+2C\bigl(0,t_{2},...,t_{n}\bigr)\int^{t_{1}+h_{1}}_{0}\bigl(t_{1}+h_{1}-s_{1}\bigr)C\bigl(s_{1},0,...,0\bigr)y\, ds_{1},
\end{align*}  
provided that $|t_{i}|<\tau_{i}$ for $1\leq i\leq n$ and $|t_{1}+h_{1}|<\tau_{1}.$ This implies
\begin{align*} &
\frac{u_{1}\bigl(t+h_{1}e_{1})-u_{1}\bigl(t\bigr) }{h_{1}}
\\& =2C\bigl(0,t_{2},...,t_{n}\bigr)\int^{t_{1}+h_{1}}_{0}C\bigl(s_{1},0,...,0\bigr)y\, ds_{1}\\&
+\frac{2}{h_{1}}C\bigl(0,t_{2},...,t_{n}\bigr)\int^{t_{1}+h_{1}}_{t_{1}}\bigl(t_{1}-s_{1}\bigr)C\bigl(s_{1},0,...,0\bigr)y\, ds_{1},
\end{align*}
provided that $|t_{i}|<\tau_{i}$ for $1\leq i\leq n$ and $|t_{1}+h_{1}|<\tau_{1}.$ Letting $h_{1}\rightarrow 0,$ we get:
\begin{align*}
\frac{\partial u_{1}}{\partial t_{1}}(t)=2C\bigl(0,t_{2},...,t_{n}\bigr)\int^{t_{1}}_{0}C\bigl(s_{1},0,...,0\bigr)y\, ds_{1},
\end{align*} 
provided that $|t_{i}|<\tau_{i}$ for $1\leq i\leq n$, and $(\partial u_{1}/\partial t_{1})_{t=(0,t_{2},...,t_{n})}=0$, provided that $t=(0,t_{2},...,t_{n} ) \in \Omega .$ Furthermore, $u_{1}\in C^{2}(\Omega : X)$ and
\begin{align*}
\frac{\partial^{2} u_{1}}{\partial t_{1}^{2}}(t)&=2C\bigl(0,t_{2},...,t_{n}\bigr)C\bigl(t_{1},0,...,0\bigr)y
\\& =C\bigl(t_{1},t_{2},...,t_{n}\bigr)Cy+C\bigl(t_{1},-t_{2},...,-t_{n}\bigr)Cy\in {\mathcal A}_{1}u_{1}(t),\quad t\in \Omega.
\end{align*}
The proof of (i) is thereby complete. If $n=2,$ then our assumption $C(t)C(s)=C(s)C(t)$ for all $t,\ s\in \Omega$ easily implies that $C(t-s)C=C(s-t)C$ for all $t,\ s\in \Omega$ such that $t-s\in\Omega$ so that $u_{1}=u_{2}.$ The first part of (ii) follows directly from (i), while the second part of (ii) follows from the fact that condition (D1) holds with the numbers $\alpha_{1}=\alpha_{2}=2$ ($1\leq i\leq 2$) and the argumentation contained in the proof of \cite[Theorem 8.1.23]{apsclcs}.

Let us consider now the issue (iii). It is clear that we have $u(0)=2^{n}C^{2}x;$ for the sake of better readibility, we will present the remainder of proof in the case that $n=3$ (the general case follows along the same lines). Clearly, we have{\small
\begin{align*}  &
u\bigl(t_{1},t_{2},t_{3}\bigr)=
C\bigl(t_{1},t_{2},t_{3}\bigr)Cx+C\bigl(t_{1},-t_{2},t_{3}\bigr)Cx+
C\bigl(t_{1},t_{2},-t_{3}\bigr)Cx+
C\bigl(t_{1},-t_{2},-t_{3}\bigr)Cx
\\& +
C\bigl(-t_{1},t_{2},t_{3}\bigr)Cx+C\bigl(-t_{1},-t_{2},t_{3}\bigr)Cx+
C\bigl(-t_{1},t_{2},-t_{3}\bigr)Cx+
C\bigl(-t_{1},-t_{2},-t_{3}\bigr)Cx
\\& :=I(t;x)+...+VIII(t;x),
\end{align*}}for any $t=(t_{1},t_{2},t_{3})\in \Omega.$  
We collect the terms $(I+IV)(t;x)$, $(II+III)(t;x)$, $(V+VIII)(t;x)$ and $(VI+VII)(t;x);$ the main point is that the first components in jointed terms are equal, while the second component and the third component in jointed terms are opposite in sign. 

Since $x\in  D({\mathcal A}_{1}),$ the argumentation contained in the proof of (i) shows that 
$\partial^{2}/\partial t_{1}^{2}(I+IV)(t;x)=(I+IV)(t;y)$, 
$\partial^{2}/\partial t_{1}^{2}(II+III)(t;x)=(II+III)(t;y),$ $\partial^{2}/\partial t_{1}^{2}(V+VIII)(t;x)=(V+VIII)(t;y)$ and 
$\partial^{2}/\partial t_{1}^{2}(VI+VII)(t;x)=(VI+VII)(t;y)$ for any $t\in \Omega$, where $y\in {\mathcal A}_{1}x.$ This implies $(\partial ^{2}u/\partial t_{1}^{2})(t)\in {\mathcal A}_{i}u(t),\;t \in \Omega.$ Since $x\in D({\mathcal A}_{1}) \cap D({\mathcal A}_{2}) \cap  D({\mathcal A}_{3})$, we get $u\in C^{2}(\Omega : X)$ by symmetry, and $(\partial ^{2}u/\partial t_{i}^{2})(t)\in {\mathcal A}_{i}u(t),\;t \in \Omega,$ $1\leq i\leq 3.$ By (i), we have{\small 
$$
\Bigl( \frac{\partial \, F(t;x)}{\partial t_{i}}\Bigr)_{t=(t_{1},...,t_{i-1},0,t_{i+1},...,t_{n})}=0,\ \mbox{ if }n=3\mbox{ and }t=\bigl(t_{1},...,t_{i-1},0,t_{i+1},...,t_{n} \bigr) \in \Omega ,
$$}where $F(t;x)$ denotes any of terms $(I+IV)(t;x),$ $(II+III)(t;x)$, $ (VI+VII)(t;x)$ or $(V+VIII)(t;x).$ Therefore, the initial conditions clarified in the last line of formulation of problem $(ACP_{2})''$ hold, which completes the first part of (iii). The second part of (iii) can be deduced as before. 
\end{proof}

\begin{rem}\label{nis}
\begin{itemize}
\item[(i)] Let us observe that
\begin{align*}
\sum_{(\sigma_{1},...,\sigma_{n}) \in \{ -1,1\}^{n}}C\bigl( \sigma_{1}t_{1},..., \sigma_{n}t_{n}\bigr)C^{n-1}=2^{n}C_{1}\bigl(t_{1}\bigr)\cdot ... \cdot C_{n}\bigl(t_{n}\bigr),\  t=\bigl(t_{1},... , t_{n}\bigr) \in \Omega ,
\end{align*} 
which tells us that we can simply compute the function
$$ 
\sum_{(\sigma_{1},...,\sigma_{n}) \in \{ -1,1\}^{n}}C\bigl( \sigma_{1}t_{1},..., \sigma_{n}t_{n}\bigr)C^{n-1}
$$ 
in terms of $C$-cosine functions $(C_{i} (\cdot  ))$, $1\leq i\leq n$. 
\item[(ii)] It can be simply proved that for each tuple $(\sigma_{1},...,\sigma_{n}) \in \{ -1,1\}^{n}$ the operator family 
$$
\bigl(C_{\sigma_{1},...,\sigma_{n}}(t_{1},...,t_{n})\equiv C( \sigma_{1}t_{1},..., \sigma_{n}t_{n})\bigr)_{(t_{1},...,t_{n})\in \Omega}
$$
is a $C$-cosine function as well as that the injectiveness of operator $C$ implies $C_{\sigma_{1},...,\sigma_{n}}(t_{i}e_{i})=C(t_{i}e_{i})$, provided that $-\tau_{i}<t_{i}<\tau_{i}$ and $1\leq i\leq n.$ 

Therefore, if  the $C$-cosine functions $(C_{i}'(t_{i}))_{-\tau_{i}<t_{i}<\tau_{i}}$ are given for $1\leq i\leq n,$ then a $C$-cosine function $(C(t))_{t\in \Omega}$ such that $C_{i}'(\cdot)=C_{i}(\cdot )$ for $1\leq i\leq n$ is not uniquely determined; the most simplest counterexample provides the cosine functions $C'(t_{1},t_{2}):=\cos(t_{1}+t_{2})$, $(t_{1},t_{2})\in {\mathbb R}^{2}$ and $C''(t_{1},t_{2}):=\cos(t_{1}-t_{2})$, $(t_{1},t_{2})\in {\mathbb R}^{2}$ ($n=2;$ $X={\mathbb C}$).
Also, we cannot express $C(\cdot)$ in terms of $C$-cosine functions $(C_{i}(\cdot))$, $1\leq i\leq n$.
\end{itemize}
\end{rem}

\section{Automatic extensions of multiparameter $C$-semigroups and multiparameter $C$-cosine functions}\label{subc}

In order to properly state our main result (Theorem \ref{nap}), which can be slightly generalized for locally equicontinuous $C$-semigroups and locally equicontinuous $C$-cosine functions defined on region $\Omega$ of form \eqref{me1}, we need following auxiliary result:

\begin{lem}\label{mata}
Suppose that the operator $C$ is injective, $\tau_{i}\in (0,+\infty]$ ($1\leq i \leq n$), $\Omega =[0,\tau_{1}) \times ... \times [0,\tau_{n})$ and $(T(t))_{t\in \Omega}$ 
is a $C$-semigroup. If $t,\,s,\ r,\ u,\ v \in \Omega$ and $t+s+r=u+v,$ then we have
\begin{align}\label{bjeda}
T(t)T(s)T(r)=T(u)T(v)C.
\end{align}
\end{lem}

\begin{proof}
Multiplying the both sides of \eqref{bjeda} with $C^{3n-3}$ and using the equality \eqref{345} as well as Theorem \ref{tucko1}(i)-(a), it suffices to prove the result in the one-dimensional setting. But, in this case, the equality \eqref{bjeda} immediately follows if we multiply the both sides of this equality with $C^{3}$ and apply \cite[Theorem 2.6]{wang} with $k=2.$
\end{proof}

\begin{thm}\label{nap}
Let the operator $C$ be injective, and let $\tau_{i}\in (0,+\infty]$ for $1\leq i \leq n$. 
\begin{itemize}
\item[(i)] Let $\Omega =[0,\tau_{1}) \times ... \times [0,\tau_{n}),$ and let $(T(t))_{t\in \Omega}$ 
be a locally equicontinuous $C$-semigroup. Then, for every integer $k\in {\mathbb N},$ there exists a locally equicontinuous $C^{k}$-semigroup $(T_{k}(t))_{t\in k\Omega}$ such that $T_{k}(t)=T(t)C^{k-1}$ for all $t\in \Omega.$ 
\item[(ii)] Let $\Omega =(-\tau_{1},\tau_{1}) \times ... \times (-\tau_{n},\tau_{n}),$ let $(T(t))_{t\in \Omega}$ 
be a locally equicontinuous $C$-semigroup, and let $(C(t))_{t\in \Omega}$ 
be a locally equicontinuous $C$-cosine function given by \eqref{3456}. Then, for every integer $k\in {\mathbb N},$ there exists a locally equicontinuous $C^{k}$-cosine function $(C_{k}(t))_{t\in k\Omega}$ such that $C_{k}(t)=C(t)C^{k-1}$ for all $t\in \Omega.$ 
\end{itemize}
\end{thm}

\begin{proof}
(i): We will prove the assertion by mathematical induction. If $k=1,$ then this is trivially true; suppose now that an integer $k\in {\mathbb N} \setminus \{1\}$ is fixed as well as that there exists a (local) $C^{k-1}$-cosine function $(T_{k-1}(t))_{t\in (k-1)\Omega}$ such that $T_{k-1}(t)=T(t)C^{k-2}$ for all $t\in \Omega.$ We define $(T_{k}(t))_{t\in k\Omega}$ by $T_{k}(t):=T_{k-1}(t)C$ for $t\in (k-1) \Omega$. If $t\in k\Omega \setminus (k-1)\Omega$, then we pick up a fixed real number $\delta_{t} \in (0,1)$, which continuously depends on $t,$ such that $(1-\delta_{t})t\in \Omega $ and $\delta_{t}t\in (k-1)\Omega$; in this case, we define 
\begin{align}\label{se}
T_{k}(t):=T\bigl(t-\delta_{t}t\bigr)T_{k-1}\bigl(\delta_{t}t\bigr).
\end{align} 
It is clear that $T_{k}(0)=C^{k}$; by induction, we can prove that $T_{k}(t)C=CT_{k}(t)$, $t\in k\Omega,$ $k\in {\mathbb N}.$ The strong continuity and the local equicontinuity of operator family $(T_{k}(t))_{t\in k\Omega}$ can be proved inductively with the help of local equicontinuity of $(T(t))_{t\in \Omega}$.

It remains to be proved the functional equality
\begin{align}\label{sem}
T_{k}(t)T_{k}(s)=T_{k}(t+s)C^{k},
\end{align}
provided that $k\in {\mathbb N} \setminus \{1\}$ and $t,\ s,\ t+s\in k\Omega.$ There exist the following possibilities:
\begin{itemize}
\item[(a)] $t,\ s,\ t+s\in (k-1)\Omega;$
\item[(b)] $t,\ s\in (k-1)\Omega$ and $t+s\in k\Omega \setminus (k-1) \Omega;$
\item[(c)] $t \in (k-1)\Omega$ and $s,\ t+s\in k\Omega \setminus (k-1) \Omega;$
\item[(d)] $s\in (k-1)\Omega$ and $t,\ t+s\in k\Omega \setminus (k-1) \Omega.$
\end{itemize}
If (a) holds, then we have
\begin{align*}
T_{k}(t)T_{k}(s)& =T_{k-1}(t)C T_{k-1}(s)C=T_{k-1}(t)  T_{k-1}(s)C^{2}
\\& =T_{k-1}(t+s)C^{k-1}C^{2}=T_{k-1}(t+s)C^{k+1}=T_{k}(t+s)C^{k}.
\end{align*}
If (b) holds, then we have
\begin{align*}
T_{k}(t)T_{k}(s)C^{2k-2}&=T_{k-1}(t)C^{k-1} C T_{k-1}(s)C^{k-1} C\\&=T_{k-1}\bigl(t-\delta_{t+s}t \bigr)T_{k-1}\bigl(\delta_{t+s}t \bigr)CT_{k-1}\bigl(s-\delta_{t+s}s \bigr)T_{k-1}\bigl(\delta_{t+s}s \bigr)C.
\end{align*}
Since $C$ is injective, the semigroup $T_{k-1}(\cdot)$ is commutative due to Theorem \ref{tucko1}(i)-(a) and the above implies 
\begin{align*}
T_{k}(t)T_{k}(s)C^{2k-2}&=
T_{k-1}\bigl(t-\delta_{t+s}t \bigr)CT_{k-1}\bigl(s-\delta_{t+s}s \bigr)T_{k-1}\bigl(\delta_{t+s}t \bigr)T_{k-1}\bigl(\delta_{t+s}s \bigr)C
\\& =T_{k-1}\bigl(t+s-\delta_{t+s}(t+s)\bigr)C^{k}T_{k-1}\bigl(\delta_{t+s}(t+s)\bigr)C^{k} 
\\& = T \bigl(t+s-\delta_{t+s}(t+s)\bigr)C^{2k-2}T_{k-1}\bigl(\delta_{t+s}(t+s)\bigr)C^{k} 
\\& =T_{k}(t+s)C^{k}C^{2k-2},
\end{align*}
where we have used by \eqref{se} in the last equality.
By the injectivity of $C,$ we finally get that \eqref{sem} holds.

If (c) holds, then we have
\begin{align*}
T_{k}(t)T_{k}(s) &=T_{k-1}(t)C T\bigl( s-\delta_{s}s\bigr)T_{k-1}\bigl( \delta_{s}s\bigr)
\end{align*}
and
\begin{align*}
T_{k}(t+s)C^{k}=T\bigl( t+s-\delta_{t+s}(t+s)\bigr)T_{k-1}\bigl( \delta_{t+s}(t+s)\bigr)C^{k} .
\end{align*}
Multiplying the right hand sides of the last two equalities with $C^{k-2},$ it suffices to prove that
\begin{align*}
T_{k-1}(t)  T_{k-1}\bigl( s-\delta_{s}s\bigr)T_{k-1}\bigl( \delta_{s}s\bigr)=T_{k-1}\bigl( t+s-\delta_{t+s}(t+s)\bigr)T_{k-1}\bigl( \delta_{t+s}(t+s)\bigr)C^{k-1}.
\end{align*}
But, this equality follows from the induction hypothesis and Lemma \ref{mata}.
We can similarly prove \eqref{sem} in the case that (d) holds.

(ii): Lemma \ref{mata} and the first part of Theorem \ref{nap} continue to hold for the region $\Omega =(-\tau_{1},\tau_{1}) \times ... \times (-\tau_{n},\tau_{n}),$ as easily explained. Let an  integer $k\in {\mathbb N} $ be fixed, and let $(T_{k}(t))_{t\in k\Omega}$ be a locally equicontinuous $C^{k}$-semigroup such that $T_{k}(t)=T(t)C^{k-1}$ for all $t\in \Omega.$  Then we set
$$
C_{k}(t):=\frac{1}{2}\Bigl[ T_{k}(t)+T_{k}(-t) \Bigr],\quad t\in k\Omega.
$$
By Theorem \ref{emoj}, we know that $(C_{k}(t))_{t\in k\Omega}$ is a $C^{k}$-cosine function; it is also clear that $(C_{k}(t))_{t\in k\Omega}$ is locally equicontinuous. Finally, we have
$$
C_{k}(t)=\frac{1}{2}\Bigl[ T_{k}(t)+T_{k}(-t) \Bigr]=\frac{1}{2}\Bigl[ T(t)C^{k-1}+T(-t)C^{k-1} \Bigr]=C(t)C^{k-1},\quad t\in \Omega.
$$
\end{proof}

Now we would like to propose the following open problem:

\begin{prob}\label{proba}
Let $\Omega =(-\tau_{1},\tau_{1}) \times ... \times (-\tau_{n},\tau_{n})$ and let $(C(t))_{t\in \Omega}$ 
be a locally equicontinuous $C$-cosine function. Is it true that, for every integer $k\in {\mathbb N},$ there exists a locally equicontinuous $C^{k}$-cosine function $(C_{k}(t))_{t\in k\Omega}$ such that $C_{k}(t)=C(t)C^{k-1}$ for all $t\in \Omega ?$ 
\end{prob}

In connection with this problem, we would like to emphasize a few relevant facts. We can try to prove the assertion by mathematical induction in the following way: If $k=1,$ then this is trivially true; suppose now that an integer $k\in {\mathbb N} \setminus \{1\}$ is fixed as well as that there exists a locally equicontinuous $C^{k-1}$-cosine function $(C_{k-1}(t))_{t\in (k-1)\Omega}$ such that $C_{k-1}(t)=C(t)C^{k-2}$ for all $t\in \Omega.$ We define $(C_{k}(t))_{t\in k\Omega}$ by $C_{k}(t):=C_{k-1}(t)C$ for $t\in (k-1) \Omega$. If $t\in k\Omega \setminus (k-1)\delta \Omega$, then there exists a fixed real number $\delta_{t} \in (0,1)$, which continuously depends on $t,$ such that $(1-\delta_{t})t\in \Omega,$ $\delta_{t}t\in (k-1)\Omega$ and $(2\delta_{t}-1)t\in (k-1)\Omega$; in this case, we define 
\begin{align*}
C_{k}(t):=2C\bigl(t-\delta_{t}t\bigr)C_{k-1}\bigl(\delta_{t}t\bigr)-C_{k-1}\bigl((2\delta_{t}-1)t\bigr)C.
\end{align*} 
Then it is clear that $C_{k}(0)=C^{k}$ and $C_{k}(t)=C(t)C^{k-1}$ for all $t\in \Omega$; by induction, we get that $C_{k}(t)C=CC_{k}(t)$, $t\in k\Omega,$ $k\in {\mathbb N}.$ The strong continuity and the local equicontinuity of operator family $(C_{k}(t))_{t\in k\Omega}$ follows from the local equicontinuity of $(C(t))_{t\in \Omega}$, and it remains to be proved the functional equality
\begin{align}\label{shp}
C_{k}(t+s)C^{k}x+C_{k}(t-s)C^{k}x=2C_{k}(t)C_{k}(s),\ \mbox{if }t,\ s,\ t\pm s \in k\Omega \mbox{ and } x\in X.
\end{align}
Arguing as in the proof of (i), by recognising two options for any of cases (a)-(d) considered above, the functional equality \eqref{shp} will be verified provided that the following equalities hold: 
\begin{itemize}
\item[(i)] \begin{align*} &
C (t+s) Cx+2C(t-s-r)C(r)x-C(2r-t+s)Cx=2C(t)C(s)x,\quad x\in X,
\end{align*}
provided that $t, \ s,\ t+s,\ t-s-r,\ r,\ 2r-t+s\in \Omega;$
\item[(ii)] \begin{align*} &
2C (t+s-r) C(r)x-C(2r-t-s)Cx
\\&+2C(t-s-v)C(v)x-C(2v-t+s)Cx=2C(t)C(s)x,\quad x\in X,
\end{align*}
provided that $t, \ s,\ t+s-r,\ r,\ 2r-t-s,\ t+s-v,\ v,\ 2v-t+s\in \Omega;$
\item[(iii)] \begin{align*} &
2C (t+s-r) C(r)x-C(2r-t-s)Cx+C(t-s)Cx
\\& =2C(t-s)\cdot \Bigl[ 2C(s-v)C(v)x-C(2v-s)Cx\Bigr],\quad x\in X,
\end{align*}
provided that $t-s, \ t+s-r,\ r,\ 2r-t-s,\ s-v,\ v,\ 2v-s\in \Omega;$
\item[(iv)] {\small \begin{align*} &
2C (t+s-r) C(r)Cx-C(2r-t-s)C^{2}x+ 2C (t-s-v) C(v)Cx-C(2v-t+s)C^{2}x
\\& =2C(t-s)\cdot \Bigl[ 2C(s-w)C(w)x-C(2w-s)Cx\Bigr],\quad x\in X,
\end{align*}}provided that $t+s-r,\ r,\ 2r-t-s,\ t-s-v,\ v,\ 2v-t-s,\ t-s,\ s-w,\ w,\ 2w-s\in \Omega .$
\end{itemize}

All these equalities can be simply proved in the one-dimensional setting; but, we really do not know how to prove their validities in the higher-dimensional setting.

\section{Asymptotically almost periodic solutions to abstract multiparameter Cauchy problems}\label{separ}

In connection with Theorem \ref{tucko1} and the equality \eqref{345}, we would like to make several new observations concerning the asymptotically almost periodic solutions to abstract multiparameter Cauchy problems. As already mentioned, the results of this section can be viewed of some independent interest; for simplicity, we will work in the setting of complex Banach spaces here. So, let $(X,\| \cdot \| )$ be a complex Banach space, let (D1) hold, and let us consider the operator family $(R(t))_{t\in [0,+\infty)^{n}}$ given by \eqref{prod}. \vspace{0.1cm}

1. Suppose that  the strongly continuous operator family $(R_{i}(t_{i}))_{t_{i}\geq 0}\subseteq L(X)$ is almost periodic for every $i\in {\mathbb N}_{n} ,$ i.e., for every $x\in X $ and $i\in {\mathbb N}_{n} ,$ the mapping $t_{i}\mapsto R_{i}(t_{i})x,$ $t_{i}\geq 0$ is almost periodic.
Then $\sup_{s_{i}\geq 0}\bigl\| R_{i}\bigl(s_{i}\bigr)\bigr\|<+\infty$ for every $i\in {\mathbb N}_{n} ,$ and  we have $(\tau =(\tau_{1},...,\tau_{n})\in [0,+\infty)^{n},\  t=(t_{1},...,t_{n})\in [0,+\infty)^{n},\ x\in X)
:$
\begin{align*}&
\| R(t+\tau)x-R(t)x\| 
\\& =\Bigl\|R_{1}\bigl(t_{1}+\tau_{1}\bigr)\cdot R_{2}\bigl(t_{2}+\tau_{2}\bigr)\cdot ...\cdot R_{n}\bigl(t_{n}+\tau_{n}\bigr)x-  R_{1}\bigl(t_{1}\bigr)\cdot R_{2}\bigl(t_{2}\bigr) \cdot...\cdot R_{n}\bigl(t_{n}\bigr)x\Bigr\|
\\& \leq \Bigl\| \bigl[ R_{1}\bigl(t_{1}+\tau_{1}\bigr)-R_{1}\bigl(t_{1}\bigr)\bigr] \cdot R_{2}\bigl(t_{2}+\tau_{2}\bigr)\cdot ...\cdot R_{n}\bigl(t_{n}+\tau_{n}\bigr)x\Bigr\|
\\&+ \Bigl\| R_{1}\bigl(t_{1}\bigr) \cdot \Bigl[ R_{2}\bigl(t_{2}+\tau_{2}\bigr)\cdot ...\cdot R_{n}\bigl(t_{n}+\tau_{n}\bigr)x-R_{2}\bigl(t_{2}\bigr)\cdot ...\cdot R_{n}\bigl(t_{n}\bigr)x \Bigr]\Bigr\|
\\& \leq \Biggl[\sup_{s_{2}\geq 0}\bigl\| R_{2}\bigl(s_{2}\bigr)\bigr\| \Biggr] \cdot ... \cdot \Biggl[\sup_{s_{n}\geq 0}\bigl\| R_{n}\bigl(s_{n}\bigr)\bigr\| \Biggr]  \cdot \bigl\| R_{1}\bigl(t_{1}+\tau_{1}\bigr)x-R_{1}\bigl(t_{1}\bigr)x\bigr\|
\\& +\Biggl[\sup_{s_{1}\geq 0}\bigl\| R_{1}\bigl(s_{1}\bigr)\bigr\| \Biggr] \cdot \Bigl\| R_{2}\bigl(t_{2}+\tau_{2}\bigr)\cdot ...\cdot R_{n}\bigl(t_{n}+\tau_{n}\bigr)x-R_{2}\bigl(t_{2}\bigr)\cdot ...\cdot R_{n}\bigl(t_{n}\bigr)x \Bigr\|.
\end{align*}\index{almost periodic resolvent operator family}Repeating this argument, we can prove that $(R(t))_{t\in [0,+\infty)^{n}}$ is almost periodic, i.e., for each $x\in X,$ the mapping $t\mapsto R(t)x,$ $t\in [0,+\infty)^{n}$ is almost periodic.

2. Let us consider now the situation in which (D1) holds with the operator family $(R_{i}(t_{i}))_{t_{i}\geq 0}$ being asymptotically almost periodic for every $i\in {\mathbb N}_{n} $, i.e., for every $x\in X$ and $i\in {\mathbb N}_{n} ,$ the mapping $t_{i}\mapsto R_{i}(t_{i})x,$ $t_{i}\geq 0$ is asymptotically almost periodic. Arguing as above, we can show that $(R(t))_{t\in [0,+\infty)^{n}}$ is asymptotically almost periodic in the sense that, for every $x\in X$ and $\epsilon>0$, there exist $L>0$ and $M>0$ such that for each $t_{0}\in [0,+\infty)^{n}$ the ball $B(t_{0},L)$ contains a point $\tau \in  [0,+\infty)^{n}$ such that, for every $t=(t_{1},...,t_{n})\in [0,+\infty)^{n}$ with the property that $\min(t_{1},...,t_{n})\geq M$, we have $\|R(t+\tau)x-R(t)x\| \leq \epsilon.$

Now we would like to present the following illustrative application of this result; cf. also the problems considered in \cite[Subsection
8.1.6]{apsclcs}. 
Suppose that $X:= L^{2}[0, 1]$, $\alpha_{i} \in  (0,2) \setminus \{1\}$, $\theta_{i}:=\pi -\pi \alpha_{i}/2$ and
$A_{i} := e^{i\theta_{i}}\Delta,$ where $\Delta$ denotes the Dirichlet Laplacian ($i=1,2$). Then we know that  $A_{i}$
is the integral generator of an asymptotically almost periodic $(g_{\alpha_{i}},I)$-resolvent family $(R_{i}(t_{i}))_{t_{i}\geq 0}$ for $i=1,2$ (if $\alpha_{1}=\alpha_{2}=1,$ then $(R_{1}(t_{1}))_{t_{1}\geq 0}$ and $(R_{2}(t_{2}))_{t_{2}\geq 0}$ are almost periodic, and we can illustrate the first part with a nice example); see also \cite[Subsection 8.1.3]{apsclcs}.
If $f\in X$ and $u : [0,+\infty)^{2}\rightarrow X$ is defined by $u(t_{1},t_{2}):=R_{1}(t_{1})R_{2}(t_{2})f,$ $t_{1}\geq 0,$ $t_{2}\geq 0,$ then $u(\cdot,\cdot)$ is a solution of problem\index{Dirichlet Laplacian}
$$
{\mathbf D}_{t_{1}}^{\alpha_{1}}{\mathbf D}_{t_{2}}^{\alpha_{2}}u\bigl(t_{1},t_{2},x\bigr)= e^{i(\theta_{1}+\theta_{2})}\Delta^{2}_{x} u\bigl(t_{1},t_{2},x\bigr),\quad t_{1}\geq 0,\ t_{2}\geq 0,\ x\in [0,1],
$$
and $u(\cdot,\cdot)$ is asymptotically almost periodic in the sense that for each $\epsilon>0$ there exist $L>0$ and $M>0$ such that for each $t_{0}\in [0,+\infty)^{2}$ the ball $B(t_{0},L)$ contains a point $\tau \in  [0,+\infty)^{2}$ such that, for every $t=(t_{1},t_{2})\in [0,+\infty)^{2}$ with the property that $\min(t_{1},t_{2})\geq M$, we have $\|u(t+\tau)-u(t)\| \leq \epsilon;$ here, ${\mathbf D}_{t_{1}}^{\alpha_{1}}{\mathbf D}_{t_{2}}^{\alpha_{2}}u(\cdot,\cdot)$ denotes the mixed Caputo fractional derivative of function
$u(\cdot,\cdot)$ of order $\alpha=(\alpha_{1},\alpha_{2}).$

\section{Conclusions and final remarks}\label{moze}

In this paper, we have presented several new results concerning multiparameter $C$-semigroups and abstract multiparameter Cauchy problems of first order. We have introduced and thoroughly analyzed the class of multiparameter $C$-cosine functions and its applications to abstract multiparameter Cauchy problems of second order. We have also considered automatic extensions of multiparameter $C$-semigroups and multiparameter $C$-cosine functions as well as asymptotically almost periodic type solutions to abstract multiparameter Cauchy problems. 

Let us finally emphasize that the assertions of \cite[Theorem 1.2, Corollary 1.3, Theorem 1.4]{boni}, concerning the topologically transitive and topologically mixing cosine functions, can be straightforwardly reformulated for multiparameter cosine operator functions; see \cite{apsclcs} for the notion. This can be also done for multiparameter $C$-cosine operator functions (see \cite[Theorem 3.2.26, Theorem 3.2.44]{knjigah} for some results obtained in the one-dimensional setting); more details will be given somewhere else.

\section*{Declarations}

\noindent {\bf Conflicts of Interest:} The authors have no competing interests to declare that are relevant to the content of this article.\vspace{0.2cm}

\noindent {\bf Financial Interests:} 
Marko Kosti\'c is partially supported by Ministry
of Science and Technological Development, Republic of Serbia.\vspace{0.2cm}

\noindent {\bf Data Availability:} The data supporting the findings of this research study are available from
the authors, upon reasonable request.

\end{document}